\documentclass[12pt]{amsart}

\calclayout

\let\oldtocsection=\tocsection

\let\oldtocsubsection=\tocsubsection

\let\oldtocsubsubsection=\tocsubsubsection

\renewcommand{\tocsection}[2]{\hspace{0em}\oldtocsection{#1}{#2}}
\renewcommand{\tocsubsection}[2]{\hspace{5em}\oldtocsubsection{#1}{#2}}
\renewcommand{\tocsubsubsection}[2]{\hspace{2em}\oldtocsubsubsection{#1}{#2}}

\usepackage{graphicx}
 \usepackage{amssymb}

\usepackage[colorlinks=true, pdfstartview=FitV, linkcolor=blue, citecolor=blue, urlcolor=blue]{hyperref}

\def\Xint#1{\mathchoice
{\XXint\displaystyle\textstyle{#1}}%
{\XXint\textstyle\scriptstyle{#1}}%
{\XXint\scriptstyle\scriptscriptstyle{#1}}%
{\XXint\scriptscriptstyle\scriptscriptstyle{#1}}%
\!\int}
\def\XXint#1#2#3{{\setbox0=\hbox{$#1{#2#3}{\int}$}
\vcenter{\hbox{$#2#3$}}\kern-.5\wd0}}

\def\avgint{\Xint-}

\newtheorem{theorem}{Theorem}[section]
\newtheorem{lemma}[theorem]{Lemma}
\newtheorem{prop}[theorem]{Proposition}
\newtheorem{corollary}[theorem]{Corollary}

\newtheorem{definition}[theorem]{Definition}
\newtheorem{example}[theorem]{Example}

\theoremstyle{definition}

\theoremstyle{remark}

\numberwithin{equation}{section}
 \allowdisplaybreaks

 \def\cprime{$'$}

\newcommand{\R}{\mathbb R}

\newcommand{\Z}{\mathbb Z}

\newcommand{\subRn}{{{\mathbb R}^n}}

\newcommand{\Q}{\mathcal Q}
\newcommand{\D}{\mathcal D}
\newcommand{\Ss}{\mathcal S}
\newcommand{\T}{\mathcal T}
\newcommand{\F}{\mathcal F}
\newcommand{\G}{\mathcal G}

\newcommand{\M}{\mathcal M}
\newcommand{\I}{\mathcal I}
\newcommand{\C}{\mathcal{C}}

\DeclareMathOperator{\supp}{supp}

\DeclareMathOperator{\grad}{\nabla}
\newcommand{\avg}[1]{\langle #1 \rangle}

\title[Norm inequalities for fractional integrals]{Two weight norm
  inequalities for fractional integral operators and commutators}

\author{David Cruz-Uribe, SFO}

\address{Department of Mathematics, Trinity
  College, Hartford, CT 06106, USA}
\email{david.cruzuribe@trincoll.edu}

\thanks{The author is supported by the Stewart-Dorwart faculty
  development fund at Trinity College and by NSF grant 1362425.}



\date{December 12, 2014}

\begin{document}

\maketitle

\vspace*{-0.25in}

{\small \tableofcontents}

\section{Introduction}
\label{section:introduction}

In these lecture notes we describe some
recent work on two weight norm inequalities for fractional
integral operators, also known as Riesz potentials, and for
commutators of fractional integrals.     Our point of view is strongly
influenced by the groundbreaking work on dyadic operators that led to
the proof of the $A_2$ conjecture by Hyt\"onen~\cite{hytonenP2010} and
the simplification of that proof by Lerner~\cite{Lern2012,MR3085756}.
(See also~\cite{hytonenP} for a more detailed history and bibliography
of this problem.)     Fractional integrals are of interest in
their own right and have important applications in the study of
Sobolev spaces and PDEs.  They are positive operators
and in many instances proofs are much easier for fractional integrals
than they are for Calder\'on-Zygmund singular integrals.  But as we
will see, in many
cases they are more difficult to work with, and we will give several
examples of results which are known to hold for singular integrals but
remain conjectures for fractional integrals. 

After giving some preliminary results in Section~\ref{section:prelim},
in Section~\ref{section:dyadic} we lay out the abstract theory of
dyadic grids and show how inequalities for fractional integrals and
commutators can be reduced to the study of dyadic operators.   All of
these ideas were implicit in the classical Calder\'on-Zygmund
decomposition but in recent years the essentials have been
extracted, yielding a substantially new perspective.  

In
Section~\ref{section:one-weight} we show how the dyadic
approach can be used to simplify the proof of one weight norm
inequalities for fractional integrals and commutators.   The purpose
of this digression is two-fold.  First, it provides a nice
illustration of the power of these dyadic methods, as the proofs are
markedly simpler than the classical proofs.  Second, we will use
these proofs to illustrate the technical obstacles we
will encounter in trying to prove two weight inequalities.

There are two approaches to two weight inequalities for fractional integrals:  the testing
conditions, first introduced by Sawyer~\cite{MR719674,MR930072}, and
the ``$A_p$ bump'' conditions introduced by Neugebauer~\cite{MR687633}
and P\'erez~\cite{MR1291534}.  Both approaches have their advantages.
In Section~\ref{section:testing} we consider testing conditions.  The
fundamental result we discuss is due to Lacey, Sawyer and
Uriarte-Tuero~\cite{LacSawUT2010}, but we will present a beautiful
simplification of their proof due to Hyt\"onen~\cite{hytonenP}.  We
conclude this section with a conjecture concerning testing conditions
for commutators of fractional integrals.

In Sections~\ref{section:bump} and~\ref{section:separated} we will discuss bump conditions.
Besides the work of P\'erez cited above, the contents of these
sections are based
on recent work by the author and
Moen~\cite{cruz-moen2012,MR3065302,MR3224572}.  We conclude the
last section with several open problems.  

Throughout these lecture notes we assume that the reader is familiar
with real analysis (e.g., as presented by Royden~\cite{MR1013117}) and
with classical harmonic analysis including the basics of the theory of
Muckenhoupt $A_p$ weights and one weight norm inequalities (e.g., the
first seven chapters of Duoandikoetxea~\cite{duoandikoetxea01}).
Additional references include the classic books by
Stein~\cite{MR0290095} and Garc\'\i a-Cuerva and Rubio de
Francia~\cite{garcia-cuerva-rubiodefrancia85} and the more recent
books by Grafakos~\cite{grafakos08a,grafakos08b}.  Many of the results
we give for weighted norm inequalities for fractional integrals are
scattered through the literature---there is unfortunately no single
reference for this material.  We will provide copious references
throughout, including historical ones.  Some of the material in these
notes is new and has not appeared in the literature before.

These notes are based on three lectures delivered at the 6th  International Course of
Mathematical Analysis in Andaluc\'\i a, held in Antequera, Spain,
September 8--12, 2014.  They are, however, greatly expanded to include
both new 
results and many 
details that I did not present in my lectures due to time constraints.  In addition,
I have taken this opportunity to correct some (relatively minor)
mistakes in the proofs I sketched in the lectures.  I am grateful to the
organizers for the invitation to present this work.  I would also
like to thank Kabe Moen, my
principal collaborator on fractional integrals (or Riesz potentials,
as he prefers), and Carlos P\'erez, who introduced me to bump
conditions and has shared his insights with me for many years.    It has
been a privilege to work with both of them.

\section{Preliminaries}
\label{section:prelim}

In this section we gather some essential definitions and a few
background results.  Hereafter,
we will be working in $\R^n$, and $n$ will always denote the
dimension.  We will denote constants by $C$, $c$, etc. and the value
may change at each appearance.  If necessary, we will denote the
dependence of the constants parenthetically: e.g., $C=C(n,p)$.   The
letters $P$ and $Q$ will be used to denote cubes in $\R^n$.  By a
weight we will always mean a non-negative, measurable function that is
positive on a set of positive measure.

Averages of functions will play a very important role in these notes,
so we introduce some useful notation.  Given any set $E$,
$0<|E|<\infty$, we define
\[ \avgint_E f(x)\,dx = \frac{1}{|E|}\int_E f(x)\,dx.  \]
More generally given a non-negative measure $\mu$, we define
\[ \avgint_E f(x)\,d\mu = \frac{1}{\mu(E)}\int_E f(x)\,d\mu.  \]
In other words, an average is always with respect to the measure.  If
we have a measure of the form $\sigma\,dx$, where $\sigma$ is a weight, we will write $d\sigma$, as in
$\avgint_E f\,d\sigma$, to emphasize this fact.   We will also use the
following more compact notation, particularly when the set is a cube~$Q$:
\[ \avgint_Q f(x)\,dx = \avg{f}_Q, \qquad \avgint_Q f(x)\,d\sigma =
\avg{f}_{Q,\sigma}. \]

\medskip

We now define the two operators we will be focusing on.  Given
$0<\alpha<n$ and a measurable function $f$, we define the fractional
integral operator $I_\alpha$
by
\[ I_\alpha f(x) = \int_\subRn \frac{f(y)}{|x-y|^{n-\alpha}}\,dy. \]
Given a function $b\in BMO$, the space of functions of bounded mean
oscillation, we define the commutator
\[ [b,I_\alpha]f(x) = b(x)I_\alpha f(x) - I_\alpha(bf)(x) =
\int_\subRn \big(b(x)-b(y)\big) \frac{f(y)}{|x-y|^{n-\alpha}}\,dy. \]
The fractional integral operator is classical:  it was introduced by
M.~Riesz~\cite{MR0030102}.  Commutators are more recent and were
first considered by~Chanillo~\cite{MR642611}.    The following are
some of the basic properties of these operators; unless otherwise
noted, see Stein~\cite[Chapter V]{MR0290095} for details.  

\begin{enumerate}
\setlength{\itemsep}{6pt}

\item $I_\alpha$ is a positive operator:  if $f(x)\geq 0$ a.e., then
  $I_\alpha f(x) \geq 0$.   Note, however, that $[b,I_\alpha]$ is not positive.

\item For $1<p<\frac{n}{\alpha}$, if we define $q$ by
  $\frac{1}{p}-\frac{1}{q}=\frac{\alpha}{n}$, then 
\[ I_\alpha : L^p \rightarrow L^q, \]
and for all $b\in BMO$,
\[ [b, I_\alpha] : L^p \rightarrow L^q. \]
See Chanillo~\cite{MR642611}. 

\item When $p=1$, $q=\frac{n}{n-\alpha}$, then $I_\alpha$ satisfies
  the weak type inequality 
\[ I_\alpha : L^p \rightarrow L^{q,\infty}, \]
but commutators are more singular and do not
satisfy a weak $(1,\frac{n}{n-\alpha})$ inequality.  For a
counter-example and a substitute inequality, see~\cite{cruz-uribe-fiorenza03}. 

\item We can define fractional powers of the Laplacian via the Fourier
  transform using the fractional integral operator: for all Schwartz
  functions $f$ and $0<\alpha<n$,
\[ (-\Delta)^{\frac{\alpha}{2}}f(x) = c I_\alpha f(x). \]
We also have that for all $f\in C_c^\infty$,
\[ |f(x)| \leq I_1(|\grad f|)(x). \]
\end{enumerate}

\medskip

Fractional integrals have found wide application in the study of
PDEs.  Here we mention a few results.   Recall the Sobolev embedding
theorem (see~\cite[Chapter V]{MR2424078}):  
if $f$ is contained in the Sobolev space  $W^{1,p}$, then for $1\leq
p<n$ and $p^*=\frac{np}{n-p}$,
\[  \|f\|_{L^{p^*}} \leq C\|\grad f\|_{L^p}. \]
When $p>1$ this is an immediate consequence of
the inequality relating $I_1$ and the gradient, and the strong type norm inequality for
$I_1$.  When $p=1$ it can be proved using the weak type
inequality for $I_1$ and a decomposition argument due to
Maz{\cprime}ya~\cite [p.~110]{mazya85} (see also Long and Nie~\cite{MR1187073}
and~\cite[Lemma~4.31]{MR2797562}). 

Two weight norm inequalities for $I_\alpha$ also yield
weighted Sobolev embeddings.  In particular, they can be used to prove
inequalities of the form
\[ \|f \|_{L^p(u)} \leq C\|\grad f\|_{L^p}. \]
These were introduced by Fefferman and Phong~\cite{fefferman83} in the study of the
Schr\"odinger operator.  Such inequalities can also be used to prove
that weak solutions of the elliptic equations with non-smooth
coefficients are strong solutions:  see, for example, Chiarenza and
Franciosi~\cite{MR1174821} and~\cite{dcu-km-sr14}.  For additional applications we refer to
the paper by Sawyer and Wheeden~\cite{MR1175693} and the many
references it contains.   (We remark in passing that this paper has been extremely
influential in the study of two weight norm inequalities for
fractional integrals.)

Closely related to the fractional integral operator is the fractional maximal
operator:  given $0<\alpha<n$ and
$f\in L^1_{loc}$, define
\[ M_\alpha f(x) = \sup_Q |Q|^{\frac{\alpha}{n}}\avgint_Q |f(y)|\,dy \cdot \chi_Q(x), \]
where the supremum is taken over all cubes with sides parallel to the
coordinate axes.  The fractional maximal operator was introduced by Muckenhoupt and
Wheeden~\cite{muckenhoupt-wheeden74} in order to proved one weight
norm inequalities for $I_\alpha$ via a good-$\lambda$ inequality.
This result is the analog of the one linking the Hardy-Littlewood maximal operator and
Calder\'on-Zygmund singular integrals proved by Coifman and
Fefferman~\cite{coifman-fefferman74}. 

For $1<p<\frac{n}{\alpha}$, $M_\alpha$ satisfies the same strong
$(p,q)$ inequality as $I_\alpha$.  In addition, it satisfies the upper
endpoint estimate $M_\alpha : L^\infty \rightarrow
L^{\frac{n}{\alpha}}$.   In contrast, if $f\in L^\infty$, then
$I_\alpha f$ need not be bounded, but does satisfy an exponential
integrability condition.   See, for instance,
Ziemer~\cite[Theorem~2.9.1]{MR1014685}.

Our approach to norm inequalities for the fractional integral operator
will avoid $M_\alpha$;  however, we will use it as a model operator
since it has many features in common with $I_\alpha$ but is usually easier to
work with.   We note in passing that there is an Orlicz fractional
maximal operator that plays a similar role for commutators of
fractional integrals:  see~\cite{cruz-uribe-fiorenza03}.  (This
operator also plays
a role in  the study of two weight, weak $(1,1)$ inequalities for $I_\alpha$: see Section~\ref{section:separated}.)

\section{Dyadic operators}
\label{section:dyadic}

In this section we explain the machinery of dyadic grids and dyadic
operators.  These ideas date back to the 1950's and the seminal work
of Calder\'on and Zygmund~\cite{MR0052553}, and have played a prominent
role in harmonic analysis since then.  In the past fifteen years  they have been
reformulated and
taken on a new
prominence because of their connection with the $A_2$
conjecture.  A important early presentation of this
new point of view was the lecture notes on dyadic
harmonic analysis by C.~Pereyra~\cite{pereyra01}.  As she described
them:

\begin{quote}
These notes contain what I consider are the main actors and universal
tools used in this area of mathematics.  They also contain an overview
of the classical problems that lead mathematicians to study these
objects and to develop the tools that are now considered the {\em abc}
of harmonic analysis.  The modern twist is the connection to a
parallel dyadic world where objects, statements and sometimes proofs
are simpler, but yet illuminated enough to guarantee that one can
translate them into the non-dyadic world.  
\end{quote}

The major advance since this was written was the realization that not
only could dyadic operators illuminate what was going on with their
non-dyadic counterparts, but in fact the solution of  non-dyadic problems could be
reduced to proving the corresponding results for dyadic operators.
Our understanding of this approach continues to evolve:  see for
instance, the very recent lecture notes on dyadic approximation by
Lerner and Nazarov~\cite{lerner-nazarov14}.

This philosophy of dyadic operators can be summarized by paraphrasing the title of the hit
song from Irving Berlin's 1946 musical, {\em Annie Get Your Gun}:

\smallskip

\begin{center}
\textit{\textbf{Anything you can do, I can do better (dyadicaly)!}}
\end{center}

\begin{figure}[h]
\includegraphics[width=2in]{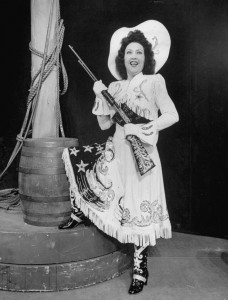}
\caption{Ethel Merman as Annie Oakley, 1946}
\end{figure}

\subsection*{Dyadic grids}
We begin by recalling the classical dyadic grid.  This is the
countable collection of cubes that are dyadic translates and dilations
of the unit cube, $[0,1)^n$:
\[ \Delta  = \{  Q = 2^k( [0,1)^n + m) :  k\in \Z, m \in \Z^n\}. \]
These cubes have a number of important properties:  any cube in
$\Delta$ has side-length a power of two; any two cubes in $\Delta$ are
disjoint or one is contained in the other; given any $k\in \Z$, the
subcollection $\Delta_k$ of cubes with side-length $2^k$ forms a
partition of $\R^n$.  

The importance of dyadic cubes lies in the Calder\'on-Zygmund cubes,
which give a very powerful decomposition of a function.    For proof
of this result, see Garc\'\i a-Cuerva and Rubio de
Francia~\cite[Chapter II]{garcia-cuerva-rubiodefrancia85}
and~\cite[Appendix~A]{MR2797562}. 

\begin{prop} \label{prop:CZcubes}
Let $f\in L^1_{loc}$ be such that $\avg{f}_Q\rightarrow 0$ as
$|Q|\rightarrow \infty$ (e.g., $f\in L^p$, $1\leq p < \infty$.)  Then
for each $\lambda>0$ there exists a collection of disjoint cubes
$\{Q_j\}\subset \Delta$ such that 
\[ \lambda < \avgint_{Q_j} |f(x)|\,dx \leq 2^n\lambda.  \]

Moreover, given $a\geq 2^{n+1}$, for each $k\in \Z$ let $\{Q_j^k\}$ be
the collection of cubes gotten by taking $\lambda = a^k$ above.  Define
\[ \Omega_k = \bigcup_j Q_j^k, \qquad E_j^k = Q_j^k \setminus
\Omega_{k+1}. \]
Then for all $j$ and $k$, the sets $E_j^k$ are pairwise disjoint and
$|E_j^k|\geq\frac{1}{2}|Q_j^k|$. 
\end{prop}

These cubes are closely related to the dyadic maximal operator:  given
$f\in L^1_{loc}$, define the operator $M^d$ (\footnote{In the notation
  we will introduce below, we would call this operator $M^\Delta$.
  Here we prefer to use the classical notation.  As Emerson said,
  ``{\em Foolish consistency is the hobgoblin of little minds.}''})
 by 
\[ M^df(x) = \sup_{Q\in \Delta} \avgint_Q |f(y)|\,dy \cdot
\chi_Q(x). \]
Then for each $\lambda>0$, if we form the cubes $Q_j$ from the first
part of Proposition~\ref{prop:CZcubes}, 
\[ \{x\in \R^n : M^df(x) > \lambda \} = \bigcup_j Q_j. \]

The Calder\'on-Zygmund cubes were introduced by Calder\'on and Zygmund
in~\cite{MR0052553}.  The essential idea underlying the sets $E_j^k$ from the second half of
Proposition~\ref{prop:CZcubes} is due to
Calder\'on~\cite{MR0442579} (working with balls in a space of
homogeneous type).  This idea was applied to Calder\'on-Zygmund cubes by
Garc\'\i a-Cuerva and Rubio de
Francia~\cite[Chapter IV]{garcia-cuerva-rubiodefrancia85} in their
proof of the reverse H\"older inequality.  
It appears to have first been explicitly
stated and proved as a property of Calder\'on-Zygmund cubes by P\'erez~\cite{MR1327936}.  

\medskip

Given the specific example of the Calder\'on-Zygmund cubes, we make the
following two definitions that extract their fundamental properties.

\begin{definition} \label{defn:dyadic-grid}
A collection of cubes $\D$ in $\R^n$ is a dyadic grid if:\
\begin{enumerate}
\setlength{\itemsep}{6pt}

\item If $Q\in \D$, then $\ell(Q)=2^k$ for some $k\in \Z$.

\item If $P,\,Q\in \D$, then $P\cap Q \in \{ P, Q, \emptyset \}$.

\item For every $k\in \Z$, the cubes $\D_k = \{ Q\in \D : \ell(Q)= 2^k
  \}$ form a partition of $\R^n$. 

\end{enumerate}
\end{definition}

\begin{definition} \label{defn:sparse}
Given a dyadic grid $\D$, a set $\Ss\subset \D$ is sparse if for every
$Q\in S$,
\[  \bigg|\bigcup_{\substack{P\in S\\P\subsetneq Q}} P\bigg|
\leq \frac{1}{2}|Q|. \]
Equivalently, if we define 
\[ E(Q) = Q \setminus \bigcup_{\substack{P\in S\\P\subsetneq Q}}
P, \]
then the sets $E(Q)$ are pairwise disjoint and $|E(Q)|\geq \frac{1}{2}|Q|$. 
\end{definition}

It is immediate that the classical dyadic cubes $\Delta$ are a dyadic grid.  By
Proposition~\ref{prop:CZcubes}, given a function $f\in L^1_{loc}$, if
we form the cubes $\{Q_j^k\}$, then they are a sparse subset of
$\Delta$ with $E(Q_j^k)=E_j^k$.   Because of this fact, given a fixed dyadic grid $\D$, we will
often refer to cubes in it as dyadic cubes.  

Clearly, we can get dyadic grids by taking translations of the
cubes in $\Delta$.  The importance of this is that every cube in
$\R^n$ is contained in a cube from a fixed, finite collection of  such dyadic grids.

\begin{theorem} \label{thm:allcubes}
There exist dyadic grids $\D^k$, $1\leq k \leq 3^n$, such that given
any cube $Q$, there exists $k$ and $P\in \D^k$ such that $Q\subset P$
and $\ell(P)\leq 3\ell(Q)$.  
\end{theorem}

The origin of Theorem~\ref{thm:allcubes} is obscure but we believe
that credit should be given to Okikiolu~\cite{MR1182488} and, for a
somewhat weaker version, to Chang, Wilson and Wolff~\cite{chang-wilson-wolff85}.(%
\footnote{Theorem~\ref{thm:allcubes} and variations of it have
  recently been attributed to Christ
in~\cite{MR1887641} and also to Garnett and
Jones in~\cite[Section~2.2]{MR2957550}.   
In particular, some people suggested that it was in the paper by Garnett and Jones on dyadic
$BMO$~\cite{MR658065}.  It is not.  Moreover, these authors have told me
and others that this result did not originate with them, though they
knew and shared it.    The earliest appearance
of a version of Theorem~\ref{thm:allcubes} in print  seems to be in Okikiolu~\cite[Lemma~1b]{MR1182488}.
Earlier, Chang,
Wilson and Wolff~\cite[Lemma~3.2]{chang-wilson-wolff85} had a weaker
but substantially similar version.  They showed that given the set $\Delta_A =
\{ Q \in \Delta, \ell(Q) \leq 2^A\}$, then there exists a finite
collection of translates of  $\Delta$ such that given any $Q\in
\Delta_A$, $3Q$ is contained in a cube of comparable size from one of
these translated grids.   A refined version of this lemma later
appeared in Wilson~\cite[Lemma~2.1]{wilson89}

The basic idea underlying the proof of Theorem~\ref{thm:allcubes} is sometimes referred to as
the ``one-third trick'' (e.g. in~\cite{MR3187852,li-pipher-ward}).
This idea has been variously  attributed~\cite{MR3187852,MR1993970} to
Garnett or Garnett and Jones, Davis,
and Wolff.  The earliest unambiguous appearance appears to be in
Wolff~\cite[Lemma~1.4]{MR659943}; Wolff attributes
this lemma to S.~Janson.})
The total number of dyadic grids needed can be reduced, though at the price
of increasing the constant $C$ relating the size of the cubes.  Hyt\"onen and P\'erez~\cite[Theorem~1.10]{MR3092729} showed that $2^n$ dyadic
grids suffice, with $C=6$.  (For details of the proof, see~\cite[Proposition~2.1]{MR3085756}.) Conde~\cite{MR2979613}
proved that only $n+1$ grids are necessary, and this bound is sharp, but with a constant
$C\approx n$.

\begin{proof}
We will use the following $3^n$ translates of the standard dyadic grid
$\Delta$:
\begin{equation} \label{eqn:grids}
 \D^t = \{ 2^j( [0,1)^n + m + t) :  j\in \Z, m \in \Z^n\}, \quad  t\in
 \{ 0, \pm 1/3\}^n.
\end{equation}
Now fix a cube $Q$; then there exists a unique $j\in \Z$ such that 
\[ \frac{2^j}{3} \leq \ell(Q) < \frac{2^{j+1}}{3}. \]
At most $2^n$ cubes in $\Delta$ of sidelength $2^j$ intersect $Q$; let $P$ be one such
that $|P\cap Q|$ is maximal.   

\begin{figure}[h]
\includegraphics[width=2in]{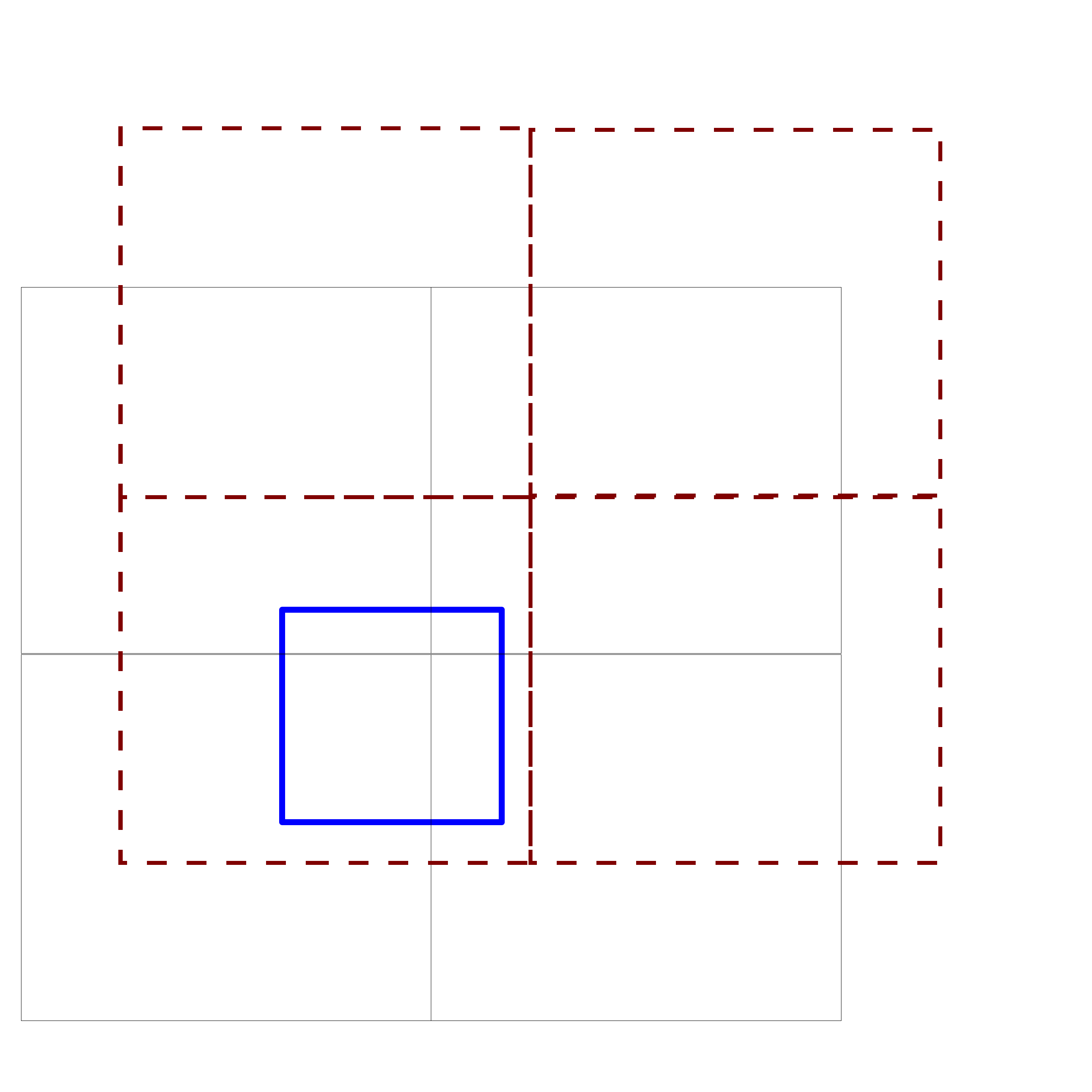}
\caption{The construction of $P'$ containing $Q$}
\end{figure}

To get the desired cube we translate $P$, acting on each coordinate in
succession.  If a face of $P$ (i.e. a $n-1$ dimensional hyper-plane on
the boundary) perpendicular to the $j$-th coordinate axis intersects
the interior of $Q$, translate $P$ parallel to the $j$-th coordinate
axis in the direction of the closest face of $Q$ a distance
$\frac{2^j}{3}$.  Because of the maximality of $P$, this direction is
away from the interior of $P$.  Hence, this moves the face out of $Q$,
and the opposite face remains outside as well, so more of $Q$ is
contained in the interior of $P$.  Thus, after at most
$n$ steps we will have a cube $P'$ that is contained in one of the
grids $\D^t$, $\ell(P')= \ell(P)\leq 3\ell(Q)$, and 
such that $Q\subset P'$.
\end{proof}  

\medskip

Though we do not consider it here, we want to note that there is
another important approach to dyadic grids.  Nazarov, Treil and
Volberg~\cite{MR1428988,MR1470373,nazarov-treil-volberg03} have shown
that random dyadic grids (i.e., translates of $\Delta$ where the
translation is taken according to some probability distribution) are
very well behaved ``on average.''  This approach was central to
Hyt\"onen's original proof of the $A_2$ conjecture~\cite{hytonenP2010}.

\subsection*{Dyadic operators}
We can now introduce the dyadic operators that we will use in place of
the fractional maximal and integral operators and commutators.  We
begin with the fractional maximal operator.   Given $0<\alpha<n$, a
dyadic grid $\D$ and $f\in L^1_{loc}$, 
define
\[ M^\D_\alpha f(x) = \sup_{Q\in \D} |Q|^{\frac{\alpha}{n}}\avgint_Q
|f(y)|\,dy \cdot \chi_Q(x). \]

\begin{prop} \label{prop:dyadic-max}
There exists a constant $C(n,\alpha)$ such that for every function
$f\in L^1_{loc}$ and $1\leq t \leq 3^n$,
\[ M_\alpha^{\D^t} f(x) \leq M_\alpha f(x) 
\leq C(n,\alpha) \sup_t  M_\alpha^{\D^t} f(x), \]
where the grids $\D^t$ are defined by~\eqref{eqn:grids}.
\end{prop}

Proposition~\ref{prop:dyadic-max} is stated in~\cite{MR3065302}
without proof; when $\alpha=0$ this was proved in~\cite[Proof of
Theorem~1.10]{MR3092729} and the proof we give for $\alpha>0$ is
essentially the same.  

\begin{proof}
The first inequality is immediate.  To prove the second, fix $x$ and a cube
$Q$ containing $x$.  Then by Theorem~\ref{thm:allcubes} there exists $t$ and $P\in
\D^t$ such that $Q\subset P$ and $|P|\leq 3^n|Q|$. Therefore,
\[ |Q|^{\frac{\alpha}{n}}\avgint_Q |f(y)|\,dy \leq 3^{n-\alpha} |P|^{\frac{\alpha}{n}}\avgint_P  |f(y)|\,dy 
\leq  C(n,\alpha) M_\alpha^{\D^t} f(x) \leq C(n,\alpha) \sup_t  M_\alpha^{\D^t} f(x). \]
If we take the supremum over all
cubes $Q$ containing $x$, we get the desired inequality.
\end{proof}

Because we are working with a finite number of dyadic grids, we have
that
\[ \sup_t  M_\alpha^{\D^t} f(x) \approx \sum_{t=1}^{3^n}
M_\alpha^{\D^t} f(x), \]
and the constants depend only on $n$.   In other words, we can
dominate any sub-linear expression for $M_\alpha$ by a sum of expressions
involving $M_\alpha^{\D^t}$.   The same will be true for~$I_\alpha$.
Hereafter, we will use this equivalence without comment.

\medskip

The dyadic analog of the fractional integral operator is defined as an
infinite sum:   given $0<\alpha<n$ and a dyadic grid $\D$, for all 
$f\in L^1_{loc}$ let
\[ I_\alpha^\D f(x) = \sum_{Q\in \D} 
|Q|^{\frac{\alpha}{n}}\avg{f}_Q \cdot \chi_Q(x).  \]
The dyadic fractional integral operator (with $\D=\Delta$) was
introduced by Sawyer and Wheeden~\cite{MR1175693} who showed that
averages over an infinite family of dyadic grids dominated
$I_\alpha$.  Here we show that only a finite number of grids is
necessary; this was proved in~\cite[Proposition~2.2]{MR3224572}.

\begin{prop} \label{prop:dyadic-frac}
There exist constants $c(n,\alpha)$, $C(n,\alpha)$ such that for every non-negative
function $f\in L^1_{loc}$ and $1\leq t \leq 3^n$,
\[ c(n,\alpha) I_\alpha^{\D^t} f(x) \leq I_\alpha f(x) 
\leq C(n,\alpha) \sup_t  I_\alpha^{\D^t} f(x), \]
where the grids $\D^t$ are defined by~\eqref{eqn:grids}.
\end{prop}

\begin{proof}
To prove the first inequality, fix a dyadic grid
  $\D=\D^t$, a non-negative function $f$, and $x\in \R^n$.    Without
  loss of generality we may assume that $f$ is bounded:  since
  $I_\alpha$ and $I_\alpha^\D$ are positive operators, the inequality
  for unbounded $f$ follows by the monotone convergence theorem.

Let
  $\{Q_k\}_{k\in \Z}\subset \D$ be the unique sequence of dyadic
  cubes such that $\ell(Q_k)=2^k$ and $x\in Q_k$.  Then for every
  integer $N>0$, 
\begin{align*}
& \sum_{\stackrel{Q\in \D}{\ell(Q)\leq 2^N}}
  |Q|^{\frac{\alpha}{n}}\avg{f}_Q\cdot \chi_Q(x) \\
& \qquad \qquad =\sum_{k=-\infty}^N|Q_k|^{\frac{\alpha}{n}-1}
\int_{Q_k\backslash Q_{k-1}} f(y)\,dy
+\sum_{k=-\infty}^N|Q_k|^{\frac{\alpha}{n}-1}\int_{Q_{k-1}} f(y)\,dy\\
& \qquad \qquad \leq c(n,\alpha) \sum_{k=-\infty}^N
\int_{Q_k\backslash Q_{k-1}}\frac{f(y)}{|x-y|^{n-\alpha}}\,dy
+2^{\alpha-n}\sum_{\stackrel{Q\in \D}{\ell(Q)\leq 2^N}} 
|Q|^{\frac{\alpha}{n}}\avg{f}_Q\cdot \chi_Q(x)\\
& \qquad \qquad =  c(n,\alpha)  \int_{Q_N} \frac{f(y)}{|x-y|^{n-\alpha}} \,dy
+ 2^{\alpha-n}\sum_{\stackrel{Q\in \D}{\ell(Q)\leq 2^N}} 
|Q|^{\frac{\alpha}{n}}\avg{f}_Q\cdot \chi_Q(x).
\end{align*}
Because $f$ is bounded, the last sum is finite.  Therefore, since $2^{\alpha-n}<1$, we can rearrange terms  and take the limit as $N\rightarrow
\infty$ to get
\[ c(n,\alpha) I^{\D}_\alpha f(x) \leq  I_\alpha f(x). \]

\medskip

To prove the second inequality, let $Q(x,r)$ be the cube of
side-length $2r$ centered at $x$.  Then 
\[ 
I_\alpha f(x)
 =\sum_{k\in \Z} \int_{Q(x,2^k)\setminus Q(x,2^{k-1})}
  \frac{f(y)}{|x-y|^{n-\alpha}}\,dy 
\leq 2^{n-\alpha}\sum_{k\in \Z} 2^{-k(n-\alpha)}\int_{Q(x,2^k)} f(y)\,dy.
\]

By Theorem~\ref{thm:allcubes}, for each $k\in \Z$ there exists
a grid $\D^t$, $1\leq t \leq 3^n$, and  $Q_t\in \D^t$ such that $Q(x,2^k)\subset Q_t$
and
\[  2^{k+1}=\ell(Q(x,2^k))\leq \ell(Q_t)\leq 6\ell(Q(x,2^k))=12\cdot 2^k.\]
Since $\ell(Q_t)=2^j$ for some $j$, we must have that  $2^{k+1}\leq
\ell(Q_t)\leq 2^{k+3}$.  Hence, 
\begin{align*}
&  2^{n-\alpha}\sum_{k\in \Z} (2^{-k})^{n-\alpha}\int_{Q(x,2^k)} f(y)\,dy \\
&\qquad \qquad \leq C(n,\alpha) \sum_{k\in \Z} \sum_{t=1}^{3^n} \sum_{\stackrel{Q\in
    \D^t}{2^{k+1}\leq \ell(Q)\leq 2^{k+3}}} 
|Q|^{\frac{\alpha}{n}}\avg{f}_Q\cdot \chi_Q(x)\\
& \qquad \qquad \leq C(n,\alpha) \sum_{t=1}^{3^n} \sum_{Q\in \D^t}  
|Q|^{\frac{\alpha}{n}}\avg{f}_Q\cdot \chi_Q(x) \\
& \qquad \qquad \leq C(n,\alpha) \sum_{t=1}^{3^n} I_\alpha^{\D^t}f(x)\\
& \qquad \qquad \leq C(n,\alpha) \sup_t I_\alpha^{\D^t}f(x). 
\end{align*}
If we combine these two estimates we get the second inequality.
\end{proof}

Intuitively, the dyadic version of the commutator $[b,I_\alpha]$ is
the operator $[b,I_\alpha^\D]$.   However, recall that this operator
is not positive:   we cannot
prove the pointwise bound 
\[  \big|[b,I_\alpha]f(x)\big| \leq C \sup_t \big|[b,I_\alpha^{\D^t}]f(x)\big|, \]
even for $f$ non-negative.   (We are not certain whether this
inequality is in fact true.)  But if we pull the absolute values
inside the integral we do get a useful dyadic
approximation of the commutator.  The following result was
implicit in~\cite{cruz-moen2012}; the proof is essentially the same as
the proof of the second inequality in Proposition~\ref{prop:dyadic-frac}.

\begin{prop} \label{prop:dyadic-commutator}
There exists a constant $C(n,\alpha)$ such that for every non-negative
function $f\in L^1_{loc}$ and $b\in BMO$,
\[  \big|[b,I_\alpha] f(x) \big|
\leq C(n,\alpha) \sup_t C_b^{\D^t}f(x), \]
where the grids $\D^t$ are defined by~\eqref{eqn:grids} and
\[ C_b^{\D^t}f(x) = \sum_{Q\in \D^t}
|Q|^{\frac{\alpha}{n}}\avgint_Q |b(x)-b(y)|f(y)\,dy \cdot \chi_Q(x). \]
\end{prop}

\medskip

\subsection*{Sparse operators}

We now come to another important reduction:  we can replace the dyadic operators
$M_\alpha^\D$ and $I_\alpha^\D$ with operators defined on sparse
families.    For the fractional maximal operator we replace it with a
linear operator that resembles the fractional integral operator.   Given a dyadic grid $\D$, a sparse
set $\Ss\subset \D$ and $f\in L_{loc}^1$,  define the operator
$L_\alpha^\Ss$ by
\[ L_\alpha^\Ss f(x) = \sum_{Q\in \Ss} |Q|^{\frac{\alpha}{n}}\avg{f}_Q
\cdot \chi_{E(Q)}(x). \]
The idea for this linearization was
implicit in Sawyer~\cite{sawyer82b}; for the maximal operator see also
de la Torre~\cite{MR1480938}.   The following result was given without
proof in ~\cite{MR3065302}.

\begin{prop} \label{prop:max-sparse}
Given a dyadic grid $\D$ and a non-negative function $f$ such that
$\avg{f}_Q \rightarrow 0$ as $|Q|\rightarrow \infty$, there
exists a sparse set $\Ss=\Ss(f) \subset \D$ and a constant $C(n,\alpha)$
independent of $f$ such that for every $x\in \R^n$,
\[ L_\alpha^\Ss f(x) \leq  M_\alpha^\D f(x) \leq C(n,\alpha)   L_\alpha^\Ss f(x). \]
\end{prop}

\begin{proof}
The sets $E(Q)$ are pairwise disjoint and for every $x\in E(Q)$, 
$|Q|^{\frac{\alpha}{n}}\avg{f}_Q \leq M_\alpha^\D f(x)$, so the first
inequality follows at once.   To prove the second inequality,  fix
$a=2^{n+1-\alpha}$ and for each $k\in \Z$,  let
\[ \Omega_k = \big\{ x\in \R^n : M_\alpha^\D f(x) > a^k \big\}. \]
For every $x\in \Omega_k$ there exists $Q\in \D$ such that
$|Q|^{\frac{\alpha}{n}}\avg{f}_Q > a^k$.  Let $\Ss_k$ be the
collection of maximal, disjoint cubes with this property.  Such
maximal cubes
exist by our assumption on $f$.  Further, by maximality we must also have that for
each $P \in \Ss_k$, $a^k<|P|^{\frac{\alpha}{n}}\avg{f}_P \leq
2^{n-\alpha}a^k$, and
\[ \Omega_k = \bigcup_{P \in \Ss_k} P. \]
Let $\Ss= \bigcup_k \Ss_k$; we claim that $\Ss$ is sparse.  Clearly
these cubes are nested:  if $P'\in \Ss_{k+1}$, then there exists $P\in
\Ss_k$ such that $P'\subsetneq P$.   Therefore, if we fix $k\in \Z$
and $P\in \Ss_k$, and  consider the union of
cubes $P'\in \Ss$ with $P'\subsetneq P$, we may restrict the union to
$P'\in \Ss_{k+1}$.  Clearly these cubes
satisfy $|P'|\leq 2^{-n}|P|$.  Hence,
\begin{multline} \label{eqn:sparse-proof}
\Big|\bigcup_{\substack{P'\in \Ss\\P'\subsetneq P}}P'\Big|
=\sum_{\substack{P'\in \Ss_{k+1}\\P'\subsetneq P}}|P'| 
< \frac{1}{a^{k+1}}\sum_{\substack{P'\in \Ss_{k+1}\\P'\subsetneq
    P}}|P'|^{\frac{\alpha}{n}}\int_{P'} f(y)\,dy \\
\leq \frac{2^{-\alpha}}{a^{k+1}}|P|^{\frac{\alpha}{n}} \int_Pf(y)\,dy\leq \frac{2^{n-2\alpha}}{a}|P|=2^{-\alpha-1}|P|.
\end{multline}

To get the desired estimate, note first that by the definition of the
cubes in $\Ss$, for each $k\in \Z$,
\[ \Omega_k \setminus \Omega_{k+1} = \bigcup_{P\in \Ss_k} E(P).  \]
Therefore, we have that for each $x\in \R^n$, there exists $k$ such
that $x\in \Omega_k \setminus \Omega_{k+1}$, and so there exists $P\in
\Ss_k$ such that
\[ M_\alpha^\D f(x) \leq a^{k+1} 
\leq a |P|^{\frac{\alpha}{n}}\avg{f}_P \cdot \chi_{E(P)} =
C(n,\alpha) \sum_{P\in \Ss} |P|^{\frac{\alpha}{n}}\avg{f}_P \cdot
\chi_{E(P)}. \]
\end{proof}

The sparse operator associated with $I_\alpha^\D$ is nearly the same
as $L_\alpha^\Ss$ except that the characteristic function is for the
entire cube $Q$.  Given a dyadic grid $\D$ and a
sparse set $\Ss\subset \D$, we define
\[ I_\alpha^\Ss f(x) = \sum_{Q\in \Ss} 
|Q|^{\frac{\alpha}{n}}\avg{f}_Q \cdot \chi_Q(x). \]
When $\alpha=0$, this operator becomes the sparse Calder\'on-Zygmund operator
that plays a central role in Lerner's proof of the $A_2$
conjecture~\cite{Lern2012,MR3085756}.  The operators $I_\alpha^\Ss$ were implicit
in Sawyer and Wheeden~\cite{MR1175693}, P\'erez~\cite{MR1291534} and
Lacey, {\em et al.}~\cite{MR2652182}, and first appeared explicitly
in~\cite{MR3224572}, where the following result was proved.

\begin{prop} \label{prop:frac-sparse}
Given a dyadic grid $\D$ and a non-negative function $f$ such that
$\avg{f}_Q \rightarrow 0$ as $|Q|\rightarrow \infty$, there
exists a sparse set $\Ss=\Ss(f) \subset \D$ and a constant $C(n,\alpha)$
independent of $f$ such that for every $x\in \R^n$,
\[ I_\alpha^\Ss f(x) \leq  I_\alpha^\D f(x) \leq C(n,\alpha)   I_\alpha^\Ss f(x). \]
\end{prop}

\begin{proof}
  The first inequality is immediate for any subset $\Ss$ of $\D$.  To
  prove the second inequality, we first construct the sparse set
  $\Ss$.  The argument is very similar to the construction in
  Proposition~\ref{prop:dyadic-max}.  Let $a=2^{n+1}$.  For each $k\in
  \Z$ define
$$\Q_k=\big\{Q\in \D:a^k<\avg{f}_Q\leq a^{k+1}\big\}.$$
Then for every $Q\in \D$ such that $\avg{f}_Q \neq 0$, there
exists a unique $k$ such that $Q\in \Q_k$.   

Now define $\Ss_k$
to be the maximal disjoint cubes contained in
\[ \big\{P\in \D:  \avg{f}_P > a^{k}\big\}. \]
Such maximal cubes exist by our hypothesis on $f$.   It follows that
given any $Q\in \Q_k$, there exists $P\in \Ss_k$ such that $Q\subset
P$.   Furthermore, these cubes are nested:  if $P'\in \Ss_{k+1}$, then
it is contained in some $P\in \Ss_k$.  If we let $\Ss= \bigcup_k
\Ss_k$, then arguing as in inequality~\eqref{eqn:sparse-proof} we have that $\Ss$ is sparse.

We now prove the desired inequality.  Fix $x\in \R^n$; then
\[ 
 I^\D_\alpha f(x)
 = \sum_k \sum_{Q\in \Q_k} |Q|^{\frac{\alpha}{n}} \avg{f}_Q \cdot \chi_Q(x) 
 \leq \sum_{k} a^{k+1}\sum_{P\in \Ss_k} 
\sum_{\substack{Q\in \Q_k\\Q\subseteq P}} |Q|^{\frac{\alpha}{n}} \cdot \chi_Q(x).
\]
The inner sum can be evaluated:
\[ 
\sum_{\substack{Q\in \Q_k\\Q\subseteq P}} |Q|^{\frac{\alpha}{n}}
\cdot\chi_Q(x)
=\sum^\infty_{r=0}\sum_{\substack{Q\in \Q_k:Q\subseteq
    P\\\ell(Q)=2^{-r}\ell(P)}}
|Q|^{\frac{\alpha}{n}}\cdot\chi_Q(x)
=\frac{1}{1-2^{-\alpha}}|P|^{\frac{\alpha}{n}}\cdot\chi_{P}(x). 
\]
Thus we have that
\begin{multline*}
 \sum_{k} a^{k+1}\sum_{P\in \Ss_k} 
\sum_{\substack{Q\in \Q_k\\Q\subseteq P}} |Q|^{\frac{\alpha}{n}} \cdot
\chi_Q(x)
\leq C(\alpha) \sum_{k} a^{k+1}\sum_{P\in
  \Ss_k}|P|^{\frac{\alpha}{n}}\cdot\chi_{P}(x) \\
\leq C(n,\alpha) \sum_k \sum_{P\in
  \Ss_k}|P|^{\frac{\alpha}{n}}\avg{f}_P\cdot\chi_{P}(x) 
= C(n,\alpha) I_\alpha^\Ss f(x).
\end{multline*}
If we combine these estimates we get the desired inequality.
\end{proof}

\bigskip

We conclude this section with a key observation: 

\begin{quote} \textbf{In light of 
Propositions~\ref{prop:dyadic-max}, \ref{prop:dyadic-frac},
\ref{prop:max-sparse} and~\ref{prop:frac-sparse},  when proving
necessary and/or sufficient conditions for weighted norm inequalities
for  fractional maximal  or integral
operators, it suffices to prove the analogous inequalities for either
the associated dyadic or sparse operators.}   
\end{quote}
In the subsequent
sections we will use this fact repeatedly.   The ability to pass to a
dyadic operator will considerably simplify the proofs.  The choice to use the dyadic or sparse
operator will be determined by the details of the proof.

Matters are more complicated for commutators.
It is possible to reduce estimates for the dyadic commutator, or more
precisely,  the dyadic operator $C_b^{\D}$ defined in
Proposition~\ref{prop:dyadic-commutator}, to estimates for a sum defined over
a sparse set.  However, this reduction does not yield a
pointwise inequality and is dependent on the particular result to be
proved.  For an example of this argument, we refer the reader
to~\cite[Theorem~1.6]{cruz-moen2012}.
This difficultly plays a role in some of the open problems which we will discuss
below.

\section{Digression:  one weight inequalities}
\label{section:one-weight}

In this section we briefly turn away from the main topic of these notes, two weight norm
inequalities, to present some basic results on one weight norm
inequalities.
We do so for two reasons.  First, in this setting it is easier to see
the advantages of the reduction to dyadic operators; second,
a closer examination of the proofs in the one weight case will
highlight where the major obstacles will be in the two weight case.

\subsection*{The fractional maximal operator}
We first consider the fractional maximal operator.  The governing
weight class is a generalization of the Muckenhoupt $A_p$ weights, and
was introduced by Muckenhoupt and
Wheeden~\cite{muckenhoupt-wheeden74}.

\begin{definition} \label{defn:Apq}
Given $0<\alpha<n$, $1<p<\frac{n}{\alpha}$, and $q$ such that
$\frac{1}{p}-\frac{1}{q}=\frac{\alpha}{n}$, we say that a weight $w$
such that $0<w(x)<\infty$ a.e. is in $A_{p,q}$ if
\[  [w]_{A_{p,q}} = \sup_Q \left(\avgint_Q
  w^{q}\,dx\right)^{\frac{1}{q}}
\left(\avgint_Q w^{-p'}\,dx\right)^{\frac{1}{p'}}<\infty, \]
where the supremum is taken over all cubes $Q$.
When $p=1$ we say $w\in A_{1,q}$ if 
\[  [w]_{A_{1,q}} = \sup_Q \sup_{x\in Q} \left(\avgint_Q
  w^{q}\,dx\right)^{\frac{1}{q}} w(x)^{-1} < \infty. \]
\end{definition}

The $A_{1,q}$ condition is equivalent to assuming that $M_qw(x)=
M(w^q)(x)^{1/q}\leq [w]_{A_{1,q}}w(x)$, that is, $w^q \in A_1$.  (For
a proof of this when $q=1$,
see~\cite[Section~5.1]{garcia-cuerva-rubiodefrancia85}.)   More
generally, if $p>1$, we have that $w\in A_{p,q}$ if and only if $w^q\in
A_{1+\frac{q}{p'}}$; this follows at once from the definition.   By
symmetry we have that $w\in A_{p,q}$ if and only if $w^{-1} \in
A_{q',p'}$, and this is equivalent to $w^{-p'}\in
A_{1+\frac{p}{q'}}$. 

In our proofs we will keep track of the dependence on the
constant $[w]_{A_{p,q}}$; however, our proofs will not yield sharp
results.  For the exact dependence, see~\cite{cruz-moen2012,MR2652182}.  

\begin{theorem} \label{thm:weak-max-one}
Given $0<\alpha<n$, $1 \leq p<\frac{n}{\alpha}$, $q$ such that
$\frac{1}{p}-\frac{1}{q}=\frac{\alpha}{n}$, and a weight $w$, the following are
equivalent:
\begin{enumerate}
\setlength{\itemsep}{6pt}
\item $w\in A_{p,q}$;

\item for any $f\in L^p(w^p)$,
\[ \sup_{t>0} t \, w^q(\{ x\in \R^n : M_\alpha f(x) > t \})^{\frac{1}{q}}
\leq C(n,\alpha) [w]_{A_{p,q}}
\left(\int_\subRn |f(x)|^pw(x)^p\,dx\right)^{\frac{1}{p}}. \]
\end{enumerate}
\end{theorem}

The sufficiency of the $A_{p,q}$ condition was first proved
in~\cite{muckenhoupt-wheeden74}.   Our proof is basically the same as
theirs, but using the sparse operator $L_\alpha^\Ss$ obviates the need
for a covering lemma argument---this is ``hidden'' in the construction
of the sparse operator.  The necessity of the $A_{p,q}$
condition was not directly considered but was implicit in their
results for the fractional integral.  Our argument below is adapted
from the case $\alpha=0$
in~\cite[Section~5.1]{garcia-cuerva-rubiodefrancia85}. 

\begin{proof}
To show the sufficiency of the $A_{p,q}$ condition,    without loss
of generality we may assume $f$ is non-negative.  It is
straightforward to show that if the sequence $\{f_k\}$ increases
pointwise a.e. to $f$, then $M_\alpha f_k$ increases to $M_\alpha f$,
so we may assume that $f$ is bounded and has compact support.  (For
the details of this argument when $\alpha=0$,
see~\cite[Lemma~3.30]{cruz-fiorenza-book}.)   Further, It will 
suffice to fix a dyadic grid $\D$ and prove the
weak type inequality for $M_\alpha^\D$.    

We first
consider the case when $p>1$.
Fix $t>0$.  If $x\in\R^n$ is such that $M_\alpha^\D f(x)>t$, then
there exists a cube $Q\in \D$ such that
$|Q|^{\frac{\alpha}{n}}\avg{f}_Q>t$.  Let $\Q$ be the set of
maximal disjoint cubes in $\D$ with this property.  (Such cubes exist by our
assumptions on $f$.)  Then by H\"older's inequality,
\begin{align*}
& t^q\,w^q(\{ x\in \R^n : M_\alpha^\D f(x) > t \}) \\
& \qquad \qquad = \sum_{Q\in \Q} w^q (Q) \\
& \qquad \qquad \leq  \sum_{Q\in \Q} w^q (Q)
\big(|Q|^{\frac{\alpha}{n}}\avg{f}_{Q}\big)^q \\
& \qquad \qquad \leq  \sum_{Q\in \Q} |Q|^{q\frac{\alpha}{n}-q} w^q (Q) 
\left(\int_{Q} f(y)w(y)w(y)^{-1}\,dy\right)^{q} \\
& \qquad \qquad \leq  \sum_{Q\in \Q} |Q|^{q\frac{\alpha}{n}-q} w^q (Q) 
\left(\int_{Q} w(y)^{-p'}\,dy\right)^{\frac{q}{p'} }
\left(\int_{Q} f(y)^p w(y)^p\,dy\right)^{\frac{q}{p} }; \\
\intertext{by our choice of $q$, $q-q\frac{\alpha}{n}=1+\frac{q}{p'}$,
  so by the $A_{p,q}$ condition,}
& \qquad \qquad \leq [w]_{A_{p,q}}^q 
\sum_{Q\in \Q} \left(\int_{Q} f(y)^p w(y)^p\,dy\right)^{\frac{q}{p} }\\
& \qquad \qquad \leq [w]_{A_{p,q}}^q 
\left( \sum_{Q\in \Q} \int_{Q} f(y)^p w(y)^p\,dy\right)^{\frac{q}{p} }\\
& \qquad \qquad \leq [w]_{A_{p,q}}^q \left( \int_\subRn f(y)^p w(y)^p\,dy\right)^{\frac{q}{p} }.
\end{align*}
The second to last inequality holds because $\frac{q}{p} \geq 1$ and
the final inequality since the cubes in $\Q$ are pairwise disjoint by
maximality.  This completes the proof of the weak type inequality when
$p>1$.  

When $p=1$ the same proof works, omitting H\"older's
inequality and using the pointwise inequality in the $A_{1,q}$
condition.

\medskip

To prove the necessity of the $A_{p,q}$ condition, we again first consider
the case $p>1$.   Fix a cube $Q$ and let $f=w^{-p'}\chi_Q$.  Then for
$x\in Q$, $M_\alpha f(x) \geq
|Q|^{\frac{\alpha}{n}}\avg{w^{-p'}}_Q$.  Then for all
$t<|Q|^{\frac{\alpha}{n}}\avg{w^{-p'}}_Q$,  the weak type inequality
implies that
\[ t^q w^q(Q) \leq C\left(\int_Q f(x)^p w(x)^p\,dx\right)^{\frac{q}{p}}
= C|Q|^{\frac{q}{p}}\left(\avgint_Q  w(x)^{-p'}\,dx\right)^{\frac{q}{p}}. \]
Taking the supremum over all such $t$ yields
\[ |Q|^{q\frac{\alpha}{n}} \int_Q w(x)^q\,dx \left( \avgint_Q
w(x)^{-p'}\,dx\right)^q
\leq C|Q|^{\frac{q}{p}}\left(\avgint_Q
    w(x)^{-p'}\,dx\right)^{\frac{q}{p}}, \]
and rearranging terms we get the $A_{p,q}$ condition on
$Q$ with a uniform constant.  

When $p=1$ we repeat the above argument but now with $f=\chi_P$, where
$P\subset Q$ is any cube.  Then we get
\[ \avgint_Q w(x)^q\,dx \leq C\left( \avgint_P w(x)^p\,dx
\right)^{\frac{q}{p}}. \]
Let $x_0$ be a Lebesgue point of $w^p$ in $Q$, and take the limit as
$P\rightarrow \{x_0\}$; by the Lebesgue differentiation theorem we get
\[ \avgint_Q w(x)^q\,dx \leq  Cw(x_0)^q. \]
The $A_{1,q}$ condition follows at once.
\end{proof}

The weak type inequality and its proof have two consequences.  First, the proof when $p=1$, holds for all $p$ and we
can replace the cube $P$ by any measurable set $E\subset Q$.  Doing
this yields an $A_\infty$ type inequality: 
\begin{equation} \label{eqn:frac-Ainfty}
 \frac{|E|}{|Q|} \leq [w]_{A_{p,q}}\left(
  \frac{w^q(E)}{w^q(Q)}\right)^{\frac{1}{q}}.
\end{equation}
Second, though we
assumed {\em a priori} in the definition of the $A_{p,q}$ condition that
$0<w(x)<\infty$ a.e., we can use this inequality to show that  this in fact is a consequence of the weak type
inequality.    For the details of the proof when $\alpha=0$,
see~\cite[Section~5.1]{garcia-cuerva-rubiodefrancia85}.   We note in
passing that the usual $A_\infty$ condition, which exchanges the roles
of $w^q$ and Lebesgue measure in~\eqref{eqn:frac-Ainfty}, is more
difficult to prove since it also requires the reverse H\"older inequality.  

\medskip

To prove the strong type inequality we could use the fact that
$w^q\in A_{1+\frac{p'}{q}}$ implies $w^q\in
A_{1+\frac{p'}{q}-\epsilon}$ for some $\epsilon>0$ to apply
Marcinkiewicz interpolation.  This is the approach used
in~\cite{muckenhoupt-wheeden74} and it requires the reverse
H\"older inequality.  

Instead, here we are going to give a
direct proof that avoids the reverse H\"older inequality.  It is based on an argument for the Hardy-Littlewood maximal
operator due to Christ and Fefferman~\cite{MR684636} that only
uses~\eqref{eqn:frac-Ainfty}.    We also  introduce an auxiliary operator, a weighted dyadic fractional
maximal operator.   Such weighted operators when $\alpha=0$ have
played an important role in the proof of sharp constant inequalities:
see~\cite{dcu-martell-perez,lernerP,MR3085756}.  Given a non-negative Borel measure $\sigma$ and a
dyadic grid $\D$, define
\[ M_{\sigma,\alpha}^\D f(x) = \sup_{Q\in \D} 
|Q|^{\frac{\alpha}{n}} \avgint_Q |f(y)|\,d\sigma \cdot \chi_Q(x). \]
If $\alpha=0$ we simply write $M_\sigma^\D$. 

\begin{lemma} \label{lemma:wtd-max}
Given $0 \leq \alpha<n$, $1\leq p <\frac{n}{\alpha}$,  $q$ such that
$\frac{1}{p}-\frac{1}{q}=\frac{\alpha}{n}$, a dyadic grid $\D$, and a non-negative Borel
measure $\sigma$, 
\[ \sup_{t>0} t\, \sigma(\{ x \in\R^n : M_{\sigma,\alpha}^\D f(x) >
t\})^{\frac{1}{q}}
\leq \left(\int_\subRn |f(x)|^p\,d\sigma\right)^{\frac{1}{p}}. \]
Furthermore, if $p>1$, 
\[  \left(\int_\subRn M_{\sigma,\alpha}^\D f(x)^q\,d\sigma\right)^{\frac{1}{q}}
\leq C(p,q) \left(\int_\subRn
  |f(x)|^p\,d\sigma\right)^{\frac{1}{p}}. \]
\end{lemma}

\begin{proof}
The proof of the weak $(1,q)$ inequality for $M_{\sigma,\alpha}^\D$ is essentially the same as the
proof of Theorem~\ref{thm:weak-max-one}.  By H\"older's inequality we
have for any cube $Q\in \D$,
\[ \sigma(Q)^{\frac{\alpha}{n}}\avgint_Q |f(x)|\,d\sigma 
\leq \sigma(Q)^{\frac{\alpha}{n}}\left(\avgint_Q
  |f(x)|^{\frac{n}{\alpha}}\,d\sigma \right)^{\frac{\alpha}{n}} 
\leq \|f\|_{L^{\frac{n}{\alpha}}(\sigma)}, \]
which immediately implies that $M_{\sigma,\alpha}^\D :
L^{\frac{n}{\alpha}}(\sigma) \rightarrow L^\infty$. 
The strong $(p,q)$ inequality
then follows from off-diagonal Marcinkiewicz
interpolation~\cite[Chapter V, Theorem~2.4]{stein-weiss71}. 
\end{proof}

\begin{theorem} \label{thm:strong-max-one}
Given $0<\alpha<n$, $1<p<\frac{n}{\alpha}$, $q$ such that
$\frac{1}{p}-\frac{1}{q}=\frac{\alpha}{n}$, and a weight $w$, the following are
equivalent:
\begin{enumerate}
\setlength{\itemsep}{6pt}
\item $w\in A_{p,q}$;

\item for any $f\in L^p(w^p)$,
\[ \left(\int_\subRn M_\alpha f(x)^q w(x)^q\,dx\right) ^{\frac{1}{q}}
\leq C(n,\alpha,p,[w]_{A_{p,q}})
\left(\int_\subRn |f(x)|^pw(x)^p\,dx\right)^{\frac{1}{p}}. \]
\end{enumerate}
\end{theorem}

\begin{proof}
Since the strong type inequality implies the weak type inequality,
necessity follows from Theorem~\ref{thm:weak-max-one}.  To prove
sufficiency we can again assume $f$ is non-negative, bounded and has compact
support, and so it is enough to prove the strong type inequality for
$L_\alpha^\Ss f$, where $\Ss$ is any sparse subset of a dyadic grid
$\D$.

Let $\sigma=w^{-p'}$.  Since the sets $E(Q)$, $Q\in \Ss$ are disjoint, we have that
\begin{align*}
\|(L_\alpha^\Ss f)w\|_q^q
& = \sum_{Q\in \Ss} |Q|^{q\frac{\alpha}{n}}\avg{f}_Q^q w^q(E(Q)) \\
& \leq \sum_{Q\in \Ss} \big( \sigma(Q)^{\frac{\alpha}{n}}
\avg{f\sigma^{-1}}_{Q,\sigma}\big)^q
|Q|^{q\frac{\alpha}{n}-q}w^q(Q)\sigma(Q) ^{q-q\frac{\alpha}{n}} \\
& =  \sum_{Q\in \Ss} \big( \sigma(Q)^{\frac{\alpha}{n}}
\avg{f\sigma^{-1}}_{Q,\sigma}\big)^q
|Q|^{-\frac{q}{p'}-1}w^q(Q)\sigma(Q) ^{\frac{q}{p'} } \sigma(Q); \\
\intertext{by inequality~\eqref{eqn:frac-Ainfty}, the properties of
  sparse cubes, the definition of
  $A_{p,q}$ and Lemma~\ref{lemma:wtd-max},}
& \leq C([w]_{A_{p,q}}) \sum_{Q\in \Ss} \big( \sigma(Q)^{\frac{\alpha}{n}}
\avg{f\sigma^{-1}}_{Q,\sigma}\big)^q \sigma(E(Q)) \\
& \leq C([w]_{A_{p,q}}) \sum_{Q\in \Ss}
\int_{E(Q)} M_{\sigma,\alpha}^\D (f\sigma^{-1})(x)^q d\sigma \\
& \leq C([w]_{A_{p,q}}) \int_\subRn M_{\sigma,\alpha}^\D
(f\sigma^{-1})(x)^q d\sigma \\
& \leq C(p,q, [w]_{A_{p,q}}) \left(\int_\subRn f(x)^p\sigma(x)^{-p}
\sigma(x)\,dx\right)^{\frac{q}{p}} \\
& = C(p,q, [w]_{A_{p,q}}) \left(\int_\subRn f(x)^p w(x)^p\,dx\right)^{\frac{q}{p}}. 
\end{align*}
\end{proof}

The above proof has several features that we want to highlight.
First, since the sets $E(Q)$, $Q\in \Ss$ are pairwise disjoint, we are
able to pull the power $q$ inside the summation.  For dyadic
fractional integrals (even sparse ones) this is no longer the case.
As we will see below, the standard technique for avoiding this problem
is to use duality.    Second, a central obstacle is that we have a sum
over cubes $Q$ that are not themselves disjoint, so we need some way
of reducing the sum to the sum of integrals over disjoint sets.  Here
we use that the cubes in $\Ss$ are sparse, and then use the $A_\infty$
property given by inequality~\eqref{eqn:frac-Ainfty}.  In the two
weight setting we will no longer have this property.  To overcome this
we will pass to a carefully chosen subfamily of cubes that are
sparse with respect to some measure induced by the weights (e.g.,
$d\sigma$ in the proof above).  

\medskip

\subsection*{The fractional integral operator}
We now turn to one weight norm inequalities for the
fractional integral operator.  We will give a direct proof of the
strong type inequality that
appears to be new, though it draws upon ideas already in the
literature:  in particular, the two weight bump conditions for the
fractional integral due to P\'erez~\cite{MR1291534} (see
Theorem~\ref{thm:frac-joined} below).   The original proof of this result
by Muckenhoupt and Wheeden~\cite{muckenhoupt-wheeden74} used a
good-$\lambda$ inequality; another proof using sharp maximal
function estimates and extrapolation was given in~\cite{cruz-uribe-martell-perez04} (see
also~\cite[Chapter 9]{MR2797562}).   One important feature of these
approaches is that they also yield weak type inequalities for the
fractional integral.  It would be very interesting to give a proof of
the weak type inequalities using the techniques of this section as it
would shed light on several open problems:  see
Section~\ref{section:separated}.  

\begin{theorem} \label{thm:strong-frac-one}
Given $0<\alpha<n$, $1<p<\frac{n}{\alpha}$, $q$ such that
$\frac{1}{p}-\frac{1}{q}=\frac{\alpha}{n}$, and a weight $w$, the following are
equivalent:
\begin{enumerate}
\setlength{\itemsep}{6pt}
\item $w\in A_{p,q}$;

\item for any $f\in L^p(w^p)$,
\[ \left(\int_\subRn I_\alpha f(x)^q w(x)^q\,dx\right) ^{\frac{1}{q}}
\leq C(n,\alpha,p,[w]_{A_{p,q}})
\left(\int_\subRn |f(x)|^pw(x)^p\,dx\right)^{\frac{1}{p}}. \]
\end{enumerate}
\end{theorem}

\begin{proof}
By the pointwise inequality $M_\alpha^\D
f(x) \leq I_\alpha^\D f(x)$, the necessity of the $A_{p,q}$ condition
follows from Theorem~\ref{thm:strong-max-one}.

To prove sufficiency, we may assume $f$ is non-negative.
Furthermore, by the monotone convergence theorem, if $\{f_k\}$ is any sequence of functions that increases pointwise
a.e. to $f$,  then for each $x\in \R^n$, $I_\alpha f_k(x)$ increases
to $I_\alpha f(x)$. Therefore, we may also assume that $f$ is bounded and
has compact support.  Thus, 
it will suffice to prove this result for the sparse operator
$I_\alpha^\Ss$, where $\Ss$ is any sparse
subset of a dyadic grid~$\D$.

Let $v=w^q$ and $\sigma=w^{-p'}$ and estimate as follows:  there exists $g\in
L^{q'}(w^{-q'})$, $\|gw^{-1}\|_{q'}=1$, such that
\begin{align*}
\|(I_\alpha^\Ss f)w\|_q 
& = \int_\subRn I_\alpha f(x) g(x)\,dx \\
& = \sum_{Q\in \Ss} |Q|^{\frac{\alpha}{n}}\avg{f}_Q \int_Q g(x)\,dx \\
& = \sum_{Q\in \Ss}  |Q|^{\frac{\alpha}{n}-1}\sigma(Q)
v(Q)^{1-\frac{\alpha}{n}}
\avg{f\sigma^{-1}}_{Q,\sigma}
v(Q)^{\frac{\alpha}{n}}\avg{gv^{-1}}_{Q,v}. \\
\end{align*}
Since $1-\frac{\alpha}{n}=\frac{1}{p'}+\frac{1}{q}$, by the definition
of the $A_{p,q}$ condition and inequality~\eqref{eqn:frac-Ainfty}
(applied to both $v$ and $\sigma$), we have that
\[ |Q|^{\frac{\alpha}{n}-1}\sigma(Q)
v(Q)^{1-\frac{\alpha}{n}}  \leq [w]_{A_{p,q}} \sigma(Q)^{\frac{1}{p}}
v(Q)^{\frac{1}{p'}} 
\leq  C([w]_{A_{p,q}}) \sigma(E(Q))^{\frac{1}{p}}
v(E(Q))^{\frac{1}{p'}}.
\]
If we combine these two estimates,  by H\"older's
inequality and Lemma~\ref{lemma:wtd-max} we get that 
\begin{align*}
&\|(I_\alpha^\Ss f)w\|_q \\
& \qquad \leq C([w]_{A_{p,q}})  \sum_{Q\in\Ss}
\avg{f\sigma^{-1}}_{Q,\sigma}\sigma(E(Q))^{\frac{1}{p}}
v(Q)^{\frac{\alpha}{n}}\avg{gv^{-1}}_{Q,v}v(E(Q))^{\frac{1}{p'}} \\
& \qquad \leq C([w]_{A_{p,q}}) \left(\sum_{Q\in \Ss} 
\avg{f\sigma^{-1}}_{Q,\sigma}^p \sigma(E(Q))\right)^{\frac{1}{p}}
\left(\sum_{Q\in \Ss} [v(Q)^{\frac{\alpha}{n}}\avg{gv^{-1}}_{Q,v}]^{p'}
  v(E(Q))\right)^{\frac{1}{p'}} \\
& \qquad \leq C([w]_{A_{p,q}}) \left(\sum_{Q\in \Ss} 
\int_{E(Q)} M_\sigma^\D (f\sigma^{-1})(x)^p \,d\sigma\right)^{\frac{1}{p}}
\left(\sum_{Q\in \Ss} 
\int_{E(Q)} M_{\sigma,\alpha}^\D(gv^{-1})(x)^{p'}\,dv
\right)^{\frac{1}{p'}} \\
& \qquad \leq C([w]_{A_{p,q}})  \left(
\int_\subRn M_\sigma^\D (f\sigma^{-1})(x)^p \,d\sigma\right)^{\frac{1}{p}}
\left(
\int_\subRn M_{\sigma,\alpha}^\D(gv^{-1})(x)^{p'}\,dv
\right)^{\frac{1}{p'}} \\
& \qquad \leq C(p,q,[w]_{A_{p,q}})  \left(
\int_\subRn (f (x)\sigma (x)^{-1})^p \,d\sigma\right)^{\frac{1}{p}}
\left(
\int_\subRn (g (x) v (x)^{-1})^{q'}\,dv
\right)^{\frac{1}{q'}} \\
& \qquad = C(p,q,[w]_{A_{p,q}}) \|fw\|_p \|gw^{-1}\|_{q'} \\
&  \qquad = C(p,q,[w]_{A_{p,q}}) \|fw\|_p.
\end{align*}
\end{proof}

\subsection*{Commutators}
We conclude this section with the statement of the one weight norm
inequality for the commutator $[b,I_\alpha]$.  This was proved
in~\cite{cruz-moen2012} using a Cauchy integral formula technique due
to Chung, Pereyra and P\'erez~\cite{MR2869172}.  We refer the reader
there for the details of the proof.

\begin{theorem} \label{thm:strong-frac-one}
Given $0<\alpha<n$, $1<p<\frac{n}{\alpha}$, $q$ such that
$\frac{1}{p}-\frac{1}{q}=\frac{\alpha}{n}$, $b\in BMO$ and a weight
$w$, then for any $f\in L^p(w^p)$,
\[ \left(\int_\subRn [b,I_\alpha] f(x)^q w(x)^q\,dx\right) ^{\frac{1}{q}}
\leq C(n,\alpha,p,[w]_{A_{p,q}},\|b\|_{BMO})
\left(\int_\subRn |f(x)|^pw(x)^p\,dx\right)^{\frac{1}{p}}. \]
\end{theorem}

\section{Testing conditions}
\label{section:testing}

In this section we turn to our main topic: two
weight norm inequalities for fractional maximal and integral
operators and for commutators.   We will consider one of the two dominant approaches to
this problem:  the Sawyer testing conditions.

\subsection*{Two weight inequalities}
Before discussing characterizations of two weight inequalities, we
first  reformulate them in a way that works well
with arbitrary weights.    We are interested in weak and strong type
inequalities of the form
\begin{gather*}
\sup_{t>0} t \, u(\{ x\in \R^n : |Tf(x)|>t\}) ^{\frac{1}{q}}
\leq C \left(\int_\subRn |f(x)|^p u(x)\,dx \right)^{\frac{1}{p}}\\
\left(\int_\subRn |Tf(x)|^q v(x)\,dx \right)^{\frac{1}{q}}
\leq C \left(\int_\subRn |f(x)|^p v(x)\,dx \right)^{\frac{1}{p}}
\end{gather*}
where $1<p\leq q< \infty$ and $T$ is one of $M_\alpha$, $I_\alpha$,
or $[b,I_\alpha]$.   For the weak type inequality we can also
consider the (more difficult) endpoint inequality when $p=1$.  For two
weight inequalities we no longer assume that there is a
relationship among $p$, $q$ and $\alpha$.  This allows us to consider
``diagonal'' inequalities: e.g.,  $I_\alpha : L^p(v)\rightarrow
L^p(u)$.    For
this reason it is more convenient to write the weights as measures
(e.g., "$u\,dx$") rather than as ``multipliers'' as we did in the
previous section for one weight norm inequalities.   

However, there are some problems with this formulation.  For instance,
since $I_\alpha$ is self-adjoint,  a strong type inequality also
implies a dual inequality.  For instance, at least formally, the dual inequality to 
\[ I_\alpha : L^p(v) \rightarrow L^{q}(u) \]
is
\[ I_\alpha : L^{q'}(u^{1-q'}) \rightarrow L^{p'}(v^{1-p'}). \]
To make sense of this we need to assume either that $0<v(x)<\infty$
a.e. (which precludes weights that have compact support) or deal with
weights that are measurable functions but equal infinity on sets of
positive measure.   This is possible, but it requires some care to
consistently evaluate expressions of the form $0\cdot \infty$.   For a
careful discussion of the details in one particular setting,
see~\cite[Section~7.2]{MR2797562}. 

To avoid these problems we adopt a point of view first introduced by
Sawyer~\cite{MR719674,MR809769}.  We introduce a new weight
$\sigma=v^{1-p'}$ and replace $f$ by $f\sigma$; then we can restate
the weak and strong type inequalities as
\begin{gather*}
\sup_{t>0} t \, u(\{ x\in \R^n : |T(f\sigma)(x)|>t\}) ^{\frac{1}{q}}
\leq C \left(\int_\subRn |f(x)|^p \sigma(x)\,dx \right)^{\frac{1}{p}},\\
\left(\int_\subRn |T(f\sigma)(x)|^q u(x)\,dx \right)^{\frac{1}{q}}
\leq C \left(\int_\subRn |f(x)|^p \sigma(x)\,dx \right)^{\frac{1}{p}}.
\end{gather*}
With this formulation, the dual inequality becomes much more natural:
for example, for $I_\alpha$, the dual of 
\begin{equation} \label{eqn:pq}
 I_\alpha (\cdot \sigma) : L^p(\sigma) \rightarrow L^q(u) 
\end{equation}
is given by
\begin{equation} \label{eqn:pq-dual}
 I_\alpha (\cdot u) : L^{q'}(u) \rightarrow L^{p'}(\sigma). 
\end{equation}
Hereafter, in a slight abuse of terminology, we will refer to
inequalities like~\eqref{eqn:pq-dual} as the dual of~\eqref{eqn:pq}
even if the operator involved (e.g., $M_\alpha$) is not self-adjoint
or even linear. 

Another advantage of this formulation (though not one we will consider
here) is that in this form one can take $u$ and $\sigma$ to be
non-negative measures.   See for instance, Sawyer~\cite{MR719674}, or more
recently, Lacey~\cite{2013arXiv1304.5004L}.

Finally, we note in
passing that two weight inequalities when $q<p$ are much more difficult and we
will not discuss them.  For more information on such inequalities for
$I_\alpha$, we refer the reader to
Verbitsky~\cite{MR1134691} and the recent paper by
Tanaka~\cite{2013arXiv1302.4164T}.   We are not aware of any analogous
results for $M_\alpha$ or $[b,I_\alpha]$.

\subsection*{Testing conditions for fractional maximal operators}
Our first approach to characterizing the pairs of weights $(u,\sigma)$ for which
a two weight inequality hold is via testing
conditions.  The basic idea of a testing condition is to show that an
operator $T$ satisfies the strong $(p,q)$ inequality
$T(\cdot\sigma) : L^p(\sigma ) \rightarrow L^q(u)$
if and only if $T$ satisfies it when restricted to a family of test
functions:  for instance, the characteristic functions of cubes,
$\chi_Q$.    This approach to the problem is due to Sawyer, who first
proved testing conditions for maximal operators~\cite{sawyer82b}, the
Hardy operator~\cite{MR809769}, and fractional
integrals~\cite{MR719674,MR930072}.  For this reason, these are often
referred to as Sawyer testing conditions.

Testing conditions received renewed interest in the work of Nazarov,
Treil and
Volberg~\cite{nazarov-treil-volberg99,nazarov-treil-volberg08,MR2019058}; they
first made explicit the conjecture that testing conditions were
necessary and sufficient for singular integral operators, beginning
with the Hilbert transform.  (Even this case is an extremely difficult problem
which was only recently solved by Lacey, Sawyer, Shen and
Uriarte-Tuero~\cite{2013arXiv1301.4663L,2012arXiv1201.4319L}.)   They
also pointed out (see~\cite{MR2019058}) the close connection between testing conditions and
the David-Journ\'e $T1$ theorem that characterizes the boundedness of singular
integrals on $L^2$.  This was not immediately obvious in the original
formulation of the $T1$ theorem, but became clear in the version given by
Stein~\cite{stein93}.  

\medskip

We first consider the testing condition that characterizes the strong
$(p,q)$ inequality for the fractional maximal operator.  As we
noted, this was first proved by Sawyer~\cite{sawyer82b}.  Here we give
a new proof based on ideas of Hyt\"onen~\cite{hytonenP} and Lacey, {\em
  et al.}~\cite{2012arXiv1201.4319L}.   For a related proof that
avoids duality and is closer in spirit to the proof of Theorem~\ref{thm:strong-max-one}, see
Kairema~\cite{MR3058926}. 

\begin{theorem} \label{thm:max-testing}
Given $0\leq \alpha <n$, $1< p \leq q< \infty$, and a pair of
weights $(u,\sigma)$, the following are equivalent:
\begin{enumerate}
\item $(u,\sigma)$ satisfy the testing condition
\[ \M_\alpha = \sup_Q \sigma(Q)^{-1/p}\left(\int_Q M_\alpha (\chi_Q\sigma)(x)^q
u(x)\,dx\right)^{\frac{1}{q}} < \infty; \]

\item for every $f\in L^p(\sigma)$,
\[ \left(\int_\subRn M_\alpha(f\sigma)(x)^q u(x)\,dx\right)^{\frac{1}{q}}
\leq C (n,p,\alpha)\M_\alpha\left(\int_\subRn |f(x)|^p \sigma(x)\,dx\right)^{\frac{1}{p}}.  \]
\end{enumerate}
\end{theorem}

To overcome the fact that the weights $u$ and $\sigma$ need not
satisfy the $A_\infty$ condition (which was central to the proof in
the one weight case) we introduce a stopping time argument referred to as the
corona decomposition.   This technique was one of the tools introduced
into the study of the $A_2$ conjecture by Lacey, Petermichl and
Reguera~\cite{MR2657437}.  The terminology goes back to David and
Semmes~\cite{MR1113517,MR1251061}, but the construction itself seems
to have first been used by Muckenhoupt and
Wheeden~\cite{MR0447956} in one dimension, where they constructed
``principal intervals.''  (See also~\cite{cruz-uribe-martell-perez05,sawyer85}.)

Before proving Theorem~\ref{thm:max-testing} we first describe the corona construction in more general terms.
Given a fixed dyadic cube $Q_0$ in a dyadic grid $\D$, a family of
dyadic cubes $\T\subset\D$ all
contained in $Q_0$, a non-negative,
locally integrable function $f$, and a
 weight $\sigma$,  we define a subfamily $\F \subset \T$
inductively.  Let $\F_0=\{Q_0\}$.   For $k\geq 0$, given the
collection of cubes $\F_k$, and $F\in \F_k$  let $\eta_\F(F)$ be the
collection of maximal disjoint subcubes $Q$ of $F$ such that
$\avg{f}_{Q,\sigma} > 2 \avg{f}_{F,\sigma}$.  (This collection could
be empty; if it is the construction stops.)
Then set
\[ \F_{k+1} = \bigcup_{F\in \F_k} \eta_\F(F) \]
and define
\[ \F = \bigcup_k \F_k.  \]
We will refer to $\F$ as the corona cubes of $f$ with respect to
$\sigma$. 

Given any cube $Q \in \T$, then by construction it is contained in
some cube in $\F$.  Let $\pi_\F(Q)$ be the smallest cube in $\F$ such
that $Q\subset \pi_\F(Q)$.  We will refer to the cubes $\eta_\F(F)$ as
the {\em children} of $F$ in $\F$, and $\pi_\F(Q)$ as the {\em parent} of $Q$
in $\F$.(\footnote{In the literature, the notation $ch_\F(F)$ is often
  used for the children of $F$.  We wanted to use Greek letters to
  denote both sets.  The letter $\eta$ seemed appropriate since it is
  the Greek ``h'', and  in Spanish the cubes in these collections are called {\em hijos} and
{\em padres.}})

The cubes in $\F$ have the critical property that they are sparse with
respect to the measure $d\sigma$.  Given any $F\in \F$, if we compute the
measure of the children of $F$ we see that
\[ \sum_{F'\in \eta_\F(F)} \sigma(F')  
\leq \frac{1}{2}\sum_{F'\in \eta_\F(F)}
\frac{(f\sigma)(F')}{\avg{f}_{\sigma,F}}
\leq \frac{1}{2}\frac{(f\sigma)(F)}{\avg{f}_{\sigma,F}}
\leq \frac{1}{2}\sigma(F). \]
Therefore, if we define the set
\[ E_\F(F) =  F \setminus \bigcup_{F' \in \eta_\F(F)} F', \]
then
\[ \sigma\big(E_\F(F)\big) \geq \frac{1}{2}\sigma(F). \]
We will refer to this as the $A_\infty$ property of the cubes in $\F$.

Below we will perform this construction not just on a single
cube $Q_0$ but on each cube in a fixed set of disjoint cubes.  We
will again refer to the collection of all the cubes that result from this
construction applied to each cube in this set as $\F$.

\begin{proof}[Proof of Theorem~\textup{\ref{thm:max-testing}}]
The necessity of the testing condition is immediate if we take
$f=\chi_Q$.   

To prove the sufficiency of the testing condition, first note that
arguing as we did in the proof of Theorem~\ref{thm:weak-max-one} we
may assume that $f$ is non-negative, bounded and has compact support.
Therefore, it will suffice to show that given any dyadic grid $\D$ and sparse set
$\Ss\subset \D$, the strong type inequality holds for $L_\alpha^\Ss$
assuming the testing condition holds for $L_\alpha^\Ss$.   Here we
use the fact that given $f$ there exists a sparse subset $\Ss$ such
that $M_\alpha^\D f(x) \lesssim L_\alpha^\Ss f(x)$, and that for every
such sparse set,  $L_\alpha^\Ss(\chi_Q\sigma)(x) \leq
M_\alpha^\D(\chi_Q \sigma)(x)$.

Fix $\D$, 
$\Ss$ and $f$.   Then there exists a function
$g\in L^{q'}(u)$, $\|g\|_{L^{q'}(u)}=1$, such that
\[ \|L_\alpha^{\Ss}(f\sigma)\|_{L^{q'}(u)} 
= \int_\subRn L_\alpha^{\Ss}(f\sigma)(x) g(x) u(x)\,dx 
= \sum_{Q\in \Ss} |Q|^{\frac{\alpha}{n}}\avg{f\sigma}_{Q}\int_{E(Q)}
  g(x)u(x)\,dx. \]

To estimate the right-hand side, fix $N\geq 0$ and let $\Ss_N$ be the
maximal disjoint cubes $Q$ in $\Ss$ such that $\ell(Q)\leq 2^N$.
Then by the monotone convergence theorem it will suffice to prove that 
\[ \sum_{Q\in \Ss_N} |Q|^{\frac{\alpha}{n}}\avg{f\sigma}_{Q}\int_{E(Q)}
  g(x)u(x)\,dx \leq C(n,p,\alpha)\M_\alpha  \|f\|_{L^p(\sigma)}. \]
For each cube $Q\in \Ss_N$, form the corona decomposition of $f$ with
respect to $\sigma$.  Then we can rewrite the sum above
as
\[  \sum_{Q\in \Ss_N} |Q|^{\frac{\alpha}{n}}\avg{f\sigma}_{Q}\int_{E(Q)}
  g(x)u(x)\,dx 
= \sum_{F\in \F} \sum_{\substack{Q\in \Ss_N\\ \pi_\F(Q)=F}} |Q|^{\frac{\alpha}{n}}\avg{f\sigma}_{Q}\int_{E(Q)}
  g(x)u(x)\,dx. \] 

Fix a cube $F$ and $Q$ such that $\pi_\F(Q)=F$.  Then given any $F'\in
\eta_\F(F)$, we must have that $F'\cap Q=\emptyset$ or $F'\subsetneq
Q$.  If the latter, then, since $\Ss$ is sparse, we must have that
$F'\cap E(Q)=\emptyset$.   Therefore,
\begin{multline*}
 \int_{E(Q)} g(x)u(x)\,dx 
\\ = \int_{E(Q)\cap E_\F(F)} g(x)u(x)\,dx + 
\sum_{F' \in \eta_\F(F)}  \int_{E(Q)\cap F'} g(x)u(x)\,dx
 = \int_{E(Q)\cap E_\F(F)} g(x)u(x)\,dx.
\end{multline*}

Let $g_F(x) = g(x)\chi_{E(F)}$ and argue as follows:  by the definition of the corona cubes, the
testing condition, and H\"older's inequality,
\begin{align*}
 & \sum_{F\in \F} \sum_{\substack{Q\in \Ss_N\\ \pi_\F(Q)=F}} 
|Q|^{\frac{\alpha}{n}}\avg{f\sigma}_{Q}\int_{E(Q)} g(x)u(x)\,dx \\
& \qquad \qquad = \sum_{F\in \F} \sum_{\substack{Q\in \Ss_N\\ \pi_\F(Q)=F}} 
|Q|^{\frac{\alpha}{n}}\avg{f}_{Q,\sigma} \avg{\sigma}_Q \int_{E(Q)}
g_F(x)u(x)\,dx \\
& \qquad \qquad \leq 2 \sum_{F\in \F} \avg{f}_{F,\sigma} \sum_{\substack{Q\in
    \Ss_N\\ \pi_\F(Q)=F}} 
|Q|^{\frac{\alpha}{n}}\avg{\sigma}_Q \int_{E(Q)}
g_F(x)u(x)\,dx\\
 & \qquad \qquad \leq 2 \sum_{F\in \F} \avg{f}_{F,\sigma} 
\int_F L_\alpha^\Ss(\sigma\chi_F)(x) g_F(x)u(x)\,dx \\
 & \qquad \qquad \leq 2 \sum_{F\in \F} \avg{f}_{F,\sigma} 
\|L_\alpha^\Ss(\sigma\chi_F)\|_{L^{q}(u)}\|g_F\chi_\F\|_{L^{q'}(u)}\\
 & \qquad \qquad \leq 2 \M_\alpha \sum_{F\in \F} \avg{f}_{F,\sigma} 
\sigma(F)^{1/p}\|g_F\chi_\F\|_{L^{q'}(u)} \\
 & \qquad \qquad \leq 2 \M_\alpha 
\left( \sum_{F\in \F} \avg{f}_{F,\sigma}^p
  \sigma(F)\right)^{\frac{1}{p}}
\left( \sum_{F\in \F} \|g_F\chi_\F\|_{L^{q'}(u)} ^{p'}\right)^{\frac{1}{p'}}.
\end{align*}

We estimate each of these sums separately.  For the first we
use the $A_\infty$ property of cubes in $\F$ and Lemma~\ref{lemma:wtd-max}:
\begin{multline*} \left( \sum_{F\in \F} \avg{f}_{F,\sigma}^p
  \sigma(F)\right)^{\frac{1}{p}} 
\leq 2^{\frac{1}{p}} \left( \sum_{F\in \F} \avg{f}_{F,\sigma}^p
  \sigma(E_\F(F))\right)^{\frac{1}{p}} \\
\leq 2^{\frac{1}{p}} \left( \sum_{F\in \F} 
\int_{E_\F(F)} M_{\sigma}^\D f(x)^p \,d\sigma \right)^{\frac{1}{p}}  
\leq 2^{\frac{1}{p}} \left( \int_\subRn M_{\sigma}^\D f(x)^p \,d\sigma
\right)^{\frac{1}{p}} 
\leq C(n,p)\|f\|_{L^p(\sigma)}.
\end{multline*}

To estimate the second sum we use the fact that $q'\leq p'$:
\begin{multline*}
 \left( \sum_{F\in \F} \|g_F\chi_\F\|_{L^{q'}(u)}
  ^{p'}\right)^{\frac{1}{p'}}
\leq \left( \sum_{F\in \F} \|g_F\chi_\F\|_{L^{q'}(u)}
  ^{q'}\right)^{\frac{1}{q'}} \\
= \left( \int_{E_\F(F)} g(x)^{q'}u(x)\,dx\right) ^{\frac{1}{q'}}
\leq \left(\int_\subRn g(x)^{q'}u(x)\,dx\right) ^{\frac{1}{q'}} =1.
\end{multline*}
If we combine these two estimates we get the desired inequality.
\end{proof}

One consequence of this proof is that a weaker
condition on the operator is actually sufficient.   At the point we apply
the testing condition, we could replace $L_\alpha^\Ss(\sigma
\chi_F)$ with the smaller, localized operator
\[ L_{\alpha,F}^{\Ss, In}\sigma(x) 
= \sum_{\substack{Q\in \Ss \\ Q\subset F}} 
|Q|^{\frac{\alpha}{n}}\avg{\sigma}_Q \chi_{E(Q)}(x).  \]
The discarded portion of the sum contains no additional
information:  for all $x\in F$,
\[ \sum_{\substack{Q\in \Ss \\ F\subset Q}} 
|Q|^{\frac{\alpha}{n}}\avg{\sigma\chi_F}_Q \chi_{E(Q)}(x)
\leq \sigma(F) \sum_{k=1}^\infty |F|^{\frac{\alpha}{n}-1} 2^{\alpha-n} \chi_F(x)
\leq C(n,\alpha) |F|^{\frac{\alpha}{n}}\avg{\sigma}_F\chi_F(x). \]
The final characteristic function is over $F$ instead of $E(F)$, but
this yields a finite overlap and so does not substantially affect the rest
of the estimate.  We will consider such local testing conditions again
for the fractional integral operator below.

\subsection*{Testing conditions for fractional integral operators}

We now prove a testing condition theorem for fractional integrals.  If
we try to modify the proof of Theorem~\ref{thm:max-testing} we quickly
discover the main obstacle: since the sum defining $I_\alpha^\Ss$ is
over the characteristic functions $\chi_Q$ and not $\chi_{E(Q)}$, the
definition of the function $g_F$ must change.  There are
additional terms in the sum and the estimate for the norm of $g_F$
no longer works.  Another condition is required to evaluate this sum.

The need for such a condition is natural:  while a testing condition
for $I_\alpha$ is clearly necessary, Sawyer~\cite{MR719674}
constructed a counter-example showing that by itself it is not sufficient.
Motivated by work of Muckenhoupt and Wheeden~\cite{MR0447956} that suggested 
duality played a role,
Sawyer~\cite{MR930072} showed that the testing condition plust the testing condition derived from
the dual inequality for
$I_\alpha$ is necessary and sufficient.  Necessity follows
immediately:  if $I_\alpha (\cdot\sigma): L^p(\sigma) \rightarrow
L^q(u)$, then, since $I_\alpha$ is a self-adjoint linear operator, we
have that $I_\alpha (\cdot u) : L^{q'}(u) \rightarrow
L^{p'}(\sigma)$.  Moreover, it turns out that this ``dual'' testing condition is the right
one for the weak type inequality.  

\begin{theorem} \label{thm:frac-testing}
Given $0\leq \alpha <n$, $1< p \leq q< \infty$, and a pair of
weights $(u,\sigma)$, then the following are equivalent:
\begin{enumerate}

\item The testing condition
\begin{gather*}
\I_\alpha = \sup_Q \sigma(Q)^{-\frac{1}{p}}
\left(\int_Q I_\alpha(\chi_Q \sigma)(x)^qu(x)\,dx\right)^{\frac{1}{q}}
< \infty, \\
\intertext{and the dual testing condition}
\I_\alpha^* = \sup_Q u(Q)^{-\frac{1}{q'}}
\left(\int_Q I_\alpha(\chi_Qu)(x)^{p'}\sigma(x)\,dx\right)^{\frac{1}{p'}}
< \infty,
\end{gather*}
hold;

\item For all $f\in L^p(\sigma)$,
\[ \left( \int_\subRn |I_\alpha (f\sigma)(x)|^q
  u(x)\,dx\right)^{\frac{1}{q}}
\leq C(n,p,q)(\I_\alpha+\I_\alpha^*) \left( \int_\subRn |f(x)|^p
  \sigma(x)\,dx\right)^{\frac{1}{p}}. \]
\end{enumerate}

The dual testing condition
is equivalent to the weak type inequality
\[ \sup_{t>0} t\, u(\{ x\in \R^n : |I_\alpha(f\sigma)(x)|>t \}) ^{\frac{1}{q}}
\leq C(n,p,q)\I_\alpha^* \left( \int_\subRn |f(x)|^p
  \sigma(x)\,dx\right)^{\frac{1}{p}}. \]
\end{theorem}

The equivalence between the dual testing
condition and the weak type inequality has the following very deep
corollary relating the weak and strong type inequalities.

\begin{corollary} \label{cor:weak-weak-strong}
Given $0<\alpha<n$ and $1< p\leq q< \infty$, 
\[  \|I_\alpha(\cdot\sigma)\|_{L^p(\sigma)\rightarrow L^q(u)} \approx
\|I_\alpha(\cdot\sigma)\|_{L^p(\sigma)\rightarrow L^{q,\infty}(u)} +
\|I_\alpha(\cdot u)\|_{L^{q'}(u)\rightarrow L^{p',\infty}(\sigma)}. \]
\end{corollary}

It is conjectured that a similar equivalence holds for singular
integrals.  However, this is a much more difficult problem and was
only recently proved for the Hilbert transform on weighted $L^2$ by Lacey, {\em
  et al.}~\cite{2012arXiv1201.4319L}.

Theorem~\ref{thm:frac-testing} was first proved by
Sawyer~\cite{MR719674,MR930072} (see also \cite{MR1175693}).  The
proof of the weak type inequality is relatively straightforward and
readily adapts to the case of dyadic operators
(see~\cite{LacSawUT2010}).  We will omit this proof and refer the
reader to these papers.  The proof of the strong type inequality is
more difficult and even for the dyadic fractional integral
operator was initially quite complex:  see Lacey, Sawyer and
Uriarte-Tuero~\cite{LacSawUT2010}.  Recently, however, Hyt\"onen has
given a much simpler proof that relies on the corona decomposition and
which is very similar to the proof given above for the fractional maximal
operator.  Besides its elegance, this proof has the advantage that it
makes clear why two testing conditions are needed:  it provides a
means of evaluating a summation over non-disjoint cubes $Q$ instead of
over disjoint sets $E(Q)$ as we did for the fractional maximal
operator.  We give this proof below.  Another proof that takes a
somewhat different approach is due to Treil~\cite{treilP}.  

\begin{proof}[Proof of Theorem~\ref{thm:frac-testing}]
As we already discussed, the necessity of the two testing conditions
is immediate.  To prove sufficiency,  we will follow the outline of
the proof of Theorem~\ref{thm:max-testing}, highlighting the changes.

First, by arguing as we did in the
proof of Theorem~\ref{thm:strong-frac-one} we can assume that $f$ is non-negative,
bounded and has compact support.   Further,  it will suffice to prove
the strong type inequality for the dyadic operator $I_\alpha^\D$,
where $\D$ is any dyadic grid, assuming that the testing condition
holds for this operator.   (We could in fact pass to the sparse
operator $I_\alpha^\Ss$, but unlike for the fractional maximal
operator, sparseness with respect to Lebesgue measure does not
simplify the proof.)

Fix a dyadic grid $\D$ and for each $N>0$ let $\D_N$ be the collection
of dyadic cubes $Q$ in $\D$ such that $\ell(Q) \leq 2^N$.   Then by duality
and the monotone convergence theorem, it will suffice to prove that for any $g\in L^{q'}(u)$,
$\|g\|_{L^{q'}(u)}=1$, 
\[ \sum_{Q\in \D_N} |Q|^{\frac{\alpha}{n}}\avg{f\sigma}_Q \int_Q
g(x)u(x)\,dx \leq C(n,p,q)(\I_\alpha+\I_\alpha^*)\|f\|_{L^p(\sigma)}.  \]
We now form two ``parallel'' corona decompositions.  For each cube in
$\D_N$ of side-length $2^N$ form the corona decomposition of $f$ with
respect to $\sigma$; denote the union of all of these cubes by $\F$.
(Since $f$ has compact support we in fact only form a finite number
of such decompositions.)  Simultaneously, on the same cubes  form the corona
decomposition of $g$ with respect to $u$; denote the union of these
sets of cubes by $\G$.   

We now decompose the sum above as follows:
\begin{align*}
\sum_{Q\in \D_N} |Q|^{\frac{\alpha}{n}}\avg{f\sigma}_Q \int_Q
g(x)u(x)\,dx 
& = \sum_{\substack{F\in \F \\ G\in \G}} 
\sum_{\substack{Q\in \D_N\\\pi_\F(Q)=F\\ \pi_\G(Q)=G}} \\
& = \sum_{F\in \F}\sum_{\substack{G\in \G \\ G\subseteq F}} 
\sum_{\substack{Q\in \D_N\\\pi_\F(Q)=F\\ \pi_\G(Q)=G}}
+  \sum_{G\in \G}\sum_{\substack{F\in \F \\ F\subsetneq G}} 
\sum_{\substack{Q\in \D_N\\\pi_\F(Q)=F\\ \pi_\G(Q)=G}} \\
& = \Sigma_1 + \Sigma_2.
\end{align*}

\medskip

We first estimate $\Sigma_1$.
Fix $F$, $G\subset F$ and $Q$ such that $\pi_\F(Q)=F$ and $\pi_\G(Q)=G$. (If no such $G$ or
$Q$ exists, then this term in the sum is vacuous and can be
disregarded.) Let $F' \in \eta_\F(F)$ be such that $Q\cap F'\neq
\emptyset$.  We cannot have $Q\subseteq F'$, since this would imply
that $\pi_\F(Q)\subseteq F' \subsetneq F$, a contradiction.  Hence,
$F' \subsetneq Q\subset G$.   We now define the function $g_F$ by
\begin{multline*}
\int_Q g(x)u(x)\,dx = \int_{Q\cap E_\F(F)} g(x)u(x)\,dx + 
\sum_{F' \in \eta_\F(F)} \int_{Q\cap F'} g(x)u(x)\,dx \\
= \int_Q \left( g(x)\chi_{E_\F(F)} + \sum_{F' \in \eta_\F(F)} 
\avg{g}_{F',u}\chi_{F'}(x)\right)u(x)\,dx = \int_Q g_F(x)u(x)\,dx.
\end{multline*}
Moreover,  in the definition of $g_F$, the sum is over
$F'\subsetneq Q\subset G$, so we can actually restrict the sum to be over
$F'$ in the set
\[ \eta_\F^*(F) = \{ F' \in \eta_\F(F) :  \pi_\G(F')\subseteq F \}. \]
 
We can now argue as follows: by the definition of corona cubes, the testing condition and
H\"older's inequality,
\begin{align*}
\Sigma_1 
& \leq 2 \sum_{F\in \F} \avg{f\sigma^{-1}}_{F,\sigma}
\sum_{\substack{G\in \G \\ G\subseteq F}} 
\sum_{\substack{Q\in \D_N\\\pi_\F(Q)=F\\ \pi_\G(Q)=G}}
|Q|^{\frac{\alpha}{n}}\avg{\sigma}_Q \int_Q g_F(x)u(x)\,dx \\
& \leq 2 \sum_{F\in \F} \avg{f\sigma^{-1}}_{F,\sigma}
\sum_{\substack{Q\in \D \\ Q\subset F}} 
|Q|^{\frac{\alpha}{n}}\avg{\sigma}_Q \int_Q g_F(x)u(x)\,dx \\
& \leq 2 \sum_{F\in \F} \avg{f\sigma^{-1}}_{F,\sigma}
\int_F I_\alpha^\D(\chi_F \sigma)(x) g_F(x) u(x)\,dx \\
& \leq 2 \sum_{F\in \F} \avg{f\sigma^{-1}}_{F,\sigma}
\| I_\alpha^\D(\chi_F \sigma)\chi_F\|_{L^q(u)}\|g_F\|_{L^{q'}(u)} \\
&  \leq 2\I_\alpha \sum_{F\in \F} \avg{f\sigma^{-1}}_{F,\sigma}
\sigma(F)^{1/p} \|g_F\|_{L^{q'}(u)} \\
& \leq 2\I_\alpha \left(\sum_{F\in \F} \avg{f\sigma^{-1}}_{F,\sigma}^p
\sigma(F)\right)^{\frac{1}{p}}
 \left(\sum_{F\in \F} \|g_F\|_{L^{q'}(u)} ^{p'}\right)^{\frac{1}{p}}.
\end{align*}

The first sum in the last term we estimate exactly as we did in the
proof of Theorem~\ref{thm:max-testing}, getting that it is bounded by
$C(n,p)\|f\|_{L^p(\sigma)}$.   To estimate the second sum we use the
fact that $q'\leq p'$ and divide it
into two parts to get 
\begin{multline*}
\left(\sum_{F\in \F} \|g_F\|_{L^{q'}(u)} ^{p'}\right)^{\frac{1}{p}}
\leq \left(\sum_{F\in \F} \|g_F\|_{L^{q'}(u)}
  ^{q'}\right)^{\frac{1}{q'}} \\
\leq  \left(\sum_{F\in \F} \int_{E_\F(F)} g(x)^{q'}u(x)\,dx \right)^{\frac{1}{q'}}
+\left(\sum_{F\in \F}  \sum_{F'\in \eta_\F^*(F)}\avg{g}_{F',u}^{q'}u(F')   \right)^{\frac{1}{q'}}.
\end{multline*}
We again estimate the first sum as we did in the
proof of Theorem~\ref{thm:max-testing}, getting that it is bounded by
$1$.  To bound the second sum, we use the properties of the corona
cubes in $\F$ and $\G$, the definition of $\eta_\F^*(F)$, and Lemma~~\ref{lemma:wtd-max}:
\begin{align*}
\sum_{F\in \F}  \sum_{F'\in \eta_\F^*(F)}\avg{g}_{F',u}^{q'}u(F') 
& = \sum_{F\in \F} \sum_{\substack{G\in \G \\ G\subseteq F}}
\sum_{\substack{F' \in
    \eta_\F(F)\\\pi_\G(F')=G}} \avg{g}_{F',u}^{q'}u(F') \\
& \leq 2^{q'}\sum_{F\in \F} \sum_{\substack{G\in \G \\ G\subseteq F}}
\avg{g}_{G,u}^{q'} \sum_{\substack{F' \in
    \eta_\F(F)\\\pi_\G(F')=G}} u(F') \\
& \leq 2^{q'}\sum_{F\in \F} \sum_{\substack{G\in \G \\ G\subseteq F}}
\avg{g}_{G,u}^{q'} u(G) \\
& \leq 2^{q'+1} \sum_{F\in \F} \sum_{\substack{G\in \G \\ G\subseteq F}}
 \avg{g}_{G,u}^{q'} u(E_\G(G)) \\
& \leq 2^{q'+1} \sum_{F\in \F} \sum_{\substack{G\in \G \\ G\subseteq F}}
\int_{E_\G(G)} M_u^\D g(x)^{q'} u(x)\,dx \\
& \leq 2^{q'+1} \int_\subRn  M_u^\D g(x)^{q'} u(x)\,dx \\
& \leq C(q) \int_\subRn g(x)^{q'} u(x)\,dx \\
& = C(q).
\end{align*}
This completes the estimate of $\Sigma_1$.

The estimate for $\Sigma_2$ is exactly the same, exchanging the roles
of $(f,\sigma)$ and $(g,u)$ and using the dual testing condition which
yields the constant $\I_\alpha^*$.
This completes the proof.
\end{proof}

\subsection*{Local and global testing conditions}

An examination of the proof of Theorem~\ref{thm:frac-testing} shows
that we did not actually need the full testing conditions on the
operator $I_\alpha^\D$; rather, we used the following localized
testing conditions:  
\begin{gather*}
\I_{\D,in}  = \sup_Q \sigma(Q)^{-\frac{1}{p}}
\left(\int_Q I_{\alpha,Q}^{\D,in}(\chi_Q \sigma)(x)^qu(x)\,dx\right)^{\frac{1}{q}}
< \infty, \\
\I_{\D,in}^* = \sup_Q u(Q)^{-\frac{1}{q'}}
\left(\int_Q I_{\alpha,Q}^{\D,in}(\chi_Qu)(x)^{p'}\sigma(x)\,dx\right)^{\frac{1}{p'}}
< \infty,
\end{gather*}
where for $x\in Q$, 
\[  I_{\alpha,Q}^{\D,in}(\chi_Q\sigma)(x) 
=\sum_{\substack{P\in \D\\ P\subseteq Q}}
|P|^{\frac{\alpha}{n}}\avg{\sigma}_P \chi_P(x). \]
Similarly, the weak type inequality is equivalent to the dual local testing
condition (i.e., the condition that $\I_{\D,in}^*<\infty$).  
This fact is not particular to the dyadic fractional integrals:  it is
a general property of positive dyadic operators and reflects the fact
that they are, in some sense, local operators.  See Lacey~{\em et
  al.}~\cite{LacSawUT2010}. 

Somewhat surprisingly, when $p<q$ the local
testing conditions can be replaced with global testing conditions:
\begin{gather*}
\I_{\D,out}  = \sup_Q \sigma(Q)^{-\frac{1}{p}}
\left(\int_\subRn I_{\alpha,Q}^{\D,out}(\chi_Q \sigma)(x)^qu(x)\,dx\right)^{\frac{1}{q}}
< \infty, \\
\I_{\D,out}^* = \sup_Q u(Q)^{-\frac{1}{q'}}
\left(\int_\subRn I_{\alpha,Q}^{\D,out}(\chi_Qu)(x)^{p'}\sigma(x)\,dx\right)^{\frac{1}{p'}}
< \infty,
\end{gather*}
where for $x\in Q$, 
\[  I_{\alpha,Q}^{\D,out}(\chi_Qu)(x) 
=\sum_{\substack{P\in \D\\ Q\subsetneq P}}
|P|^{\frac{\alpha}{n}}\avg{\sigma\chi_Q}_P \chi_P(x). \]
We record this fact as theorem; we will discuss one of its
consequences in
Section~\ref{section:separated}.  For a proof, see ~\cite{LacSawUT2010}. 

\begin{theorem} \label{thm:lsut-global}
Given $0\leq \alpha <n$, $1< p < q< \infty$, a dyadic grid
$\D$, and a pair of weights $(u,\sigma)$, then:
\begin{enumerate}

\item $\|I_\alpha^\D(\cdot\sigma)\|_{L^p(\sigma)\rightarrow L^q(u)}
  \approx I_{\D,out} + \I_{\D,out}^*$;

\item $\|I_\alpha^\D(\cdot\sigma)\|_{L^p(\sigma)\rightarrow L^{q,\infty}(u)}
  \approx \I_{\D,out}^*$.

\end{enumerate}
\end{theorem}

\subsection*{Testing conditions for commutators}
We conclude this section by considering testing conditions and
commutators.  This problem is completely open but we give
some conjectures and also sketch some possible approaches and the problems which
will be encountered.  

In light of the testing conditions in Theorems~\ref{thm:max-testing}
and~\ref{thm:frac-testing}, it seems reasonable to conjecture that for
$1<p\leq q < \infty$, $0<\alpha<n$ and $b\in BMO$, the following two
testing conditions,
\begin{gather*}
\C_\alpha = \sup_Q \sigma(Q)^{-\frac{1}{p}}
\left(\int_Q [b,I_\alpha](\chi_Q \sigma)(x)^qu(x)\,dx\right)^{\frac{1}{q}}
< \infty, \\
\C_\alpha^* = \sup_Q u(Q)^{-\frac{1}{q'}}
\left(\int_Q [b,I_\alpha](\chi_Qu)(x)^{p'}\sigma(x)\,dx\right)^{\frac{1}{p'}}
< \infty,
\end{gather*}
are necessary and sufficient for the strong type inequality
$[b,I_\alpha](\cdot \sigma) : L^p(\sigma) \rightarrow L^q(u)$, and
that the dual testing condition (i.e., $\C_\alpha^*<\infty$) is necessary and
sufficient for the weak type inequality.    The necessity of both
testing conditions for the strong type inequality is immediate.   The
necessity of the dual testing condition follows by duality:  see, for
instance, Sawyer~\cite{MR719674} for the proof of necessity for
$I_\alpha$ which adapts immediately to this case.

A significant obstacle for proving sufficiency is that we cannot pass
directly to dyadic operators, such as the operator $C_b^\D$ defined in
Proposition~\ref{prop:dyadic-commutator}.     The first problem is
that since $[b,I_\alpha]$ is not a positive operator, we do not have
an obvious pointwise equivalence between $[B,I_\alpha]$ and $C_b^\D$.
Therefore, we cannot pass from a testing condition for the commutator
to a dyadic testing condition as we did in the proof of
Theorem~\ref{thm:frac-testing}.   This means that we will be required
to work directly with the non-dyadic testing conditions.  This is very much
the same situation as is encountered for the Hilbert transform, and we
suspect that the same (sophisticated) techniques used there may be
applicable to this problem.  In addition, the recent work of Sawyer,
{\em et al.}~\cite{2013arXiv1302.5093S} on fractional singular
integrals in higher dimensions should also be relevant.

An intermediate result would be to prove that testing conditions for
the operator $C_b^\D$ are necessary and sufficient for that operator
to be bounded, which would yield a sufficient condition for
$[b,I_\alpha]$.  In this case the parallel corona decomposition used
in the proof of Theorem~\ref{thm:frac-testing} should be applicable,
but there remain some significant technical obstacles.  In particular,
it is not clear how to use the fact that $b$ is in $BMO$ in a way
which interacts well with the corona decomposition.  

\section{Bump conditions}
\label{section:bump}

In this section we discuss the second approach to two weight norm
inequalities, the $A_p$-bump conditions.  These were first
introduced by Neugebauer~\cite{MR687633}, but they were systematically
developed by P\'erez~\cite{MR1291534,MR1327936}.   They are a generalization
of the Muckenhoupt $A_p$ and Muckenhoupt-Wheeden $A_{p,q}$
conditions.    Compared to testing conditions they have several
relative strengths and weaknesses.  They only provide
sufficient conditions---they are not necessary, though examples show
that they are in some sense sharp (see~\cite{MR1713140}).  On the other hand,
they are ``universal'' sufficient conditions:  they give conditions
that hold for families of operators and are not conditioned to
individual operators.  (This property is  much more important in the
study of singular integrals than it is for the study of fractional
integrals.)   The bump conditions are geometric conditions on the
weights and do not involve the operator, so in practice it is easier
to check whether a pair of weights satisfies a bump condition.   In
addition, there exists a very flexible technique for constructing pairs
that satisfy a given condition:  the method of factored weights which
we will discuss below.  Finally, since the bump conditions are defined
with respect to cubes, they work well with the Calder\'on-Zygmund
decomposition and with dyadic grids in general.  

\subsection*{The $A_{p,q}^\alpha$ condition}

We begin by defining the natural generalization of the one weight
$A_{p,q}$ condition given in Definition~\ref{defn:Apq}.   To state it
we introduce the following notation for normalized, localized $L^p$
norms:  given $1\leq p<\infty$ and a cube $Q$,
\[ \|f\|_{p,Q} = \left(\avgint_Q |f(x)|^p\,dx\right)^{\frac{1}{p}}. \]

\begin{definition}  \label{eqn:two-wt-Apq}
Given $1<p\leq q< \infty$ and $0\leq \alpha<n$, we say that a pair of
weights $(u,\sigma)$ is in the class $A_{p,q}^\alpha$ if
\[ [u,\sigma]_{A_{p,q}^\alpha} = \sup_Q
|Q|^{\frac{\alpha}{n}+\frac{1}{q}-\frac{1}{p}}
\|u^{\frac{1}{q}}\|_{q,Q}\|\sigma^{\frac{1}{p'}}\|_{p',Q} < \infty. \]
\end{definition}

We can extend this definition to the case $p=1$ by using the
$L^\infty$ norm.  However, in this case it makes more sense to express
the endpoint weak type inequality in terms of pairs $(u,v)$ as
originally discussed in Section~\ref{section:testing}.  We will
consider these endpoint inequalities in
Section~\ref{section:separated}. 

\medskip

The two weight $A_{p,q}^\alpha$ characterizes weak type inequalities
for $M_\alpha$.  This result is well-known but a proof has never appeared in the
literature since it is very similar to
the proof of Theorem~\ref{thm:weak-max-one}; we also omit the details.
 For a generalization to non-homogeneous spaces whose proof
adapts well to dyadic grids, see
Garc\'\i a-Cuerva and Martell~\cite{garcia-cuerva-martell01}.

\begin{theorem} \label{thm:bump-max-weak}
Given $1<p\leq q<\infty$, $0<\alpha<n$, and a pair of weights
$(u,\sigma)$, the following are equivalent:
\begin{enumerate}

\item $(u,\sigma)\in A_{p,q}^\alpha$; 

\item for any $f\in L^p(\sigma)$,
\[ \sup_{t>0} t\, u(\{ x\in \R^n M_\alpha(f\sigma)(x)>t \}) ^{\frac{1}{q}} 
\leq C(n,\alpha)[u,\sigma]_{A_{p,q}^\alpha} 
\left(\int_\subRn |f(x)|^p\sigma(x)\,dx\right)^{\frac{1}{p}}.
\]
\end{enumerate}
\end{theorem}

\medskip

While the $A_{p,q}^\alpha$ condition characterizes the weak type inequality, it
is not sufficient for the strong type
inequality.  This fact has been part of the folklore of
the field, but a counter-example was not published until
recently~\cite{MR3065302}.   When $\alpha=0$, a counter-example for
the  Hardy-Littlewood maximal operator was constructed by Muckenhoupt and
Wheeden~\cite{muckenhoupt-wheeden76}.   However, this example does not
extend to the case $\alpha>0$ and our construction is substantially
different from theirs.  

\begin{example} \label{example:bad-wt}
Given $1<p\leq q < \infty$ and $0<\alpha<n$, 
there exists a pair of weights $(u,\sigma)\in A^\alpha_{p,q}$ and a
function $f\in L^p(\sigma)$ such that $M_\alpha(f\sigma)\not\in L^q(u)$. 
\end{example}

To construct Example~\ref{example:bad-wt} we will make use of the
technique of factored weights.   Factored weights are generalization
of the easier half of the Jones $A_p$ factorization theorem:
given $w_1,\,w_2\in A_1$, then for $1<p<\infty$, $w_1w_2^{1-p}\in
A_p$.  (See~\cite{duoandikoetxea01,garcia-cuerva-rubiodefrancia85};
in~\cite{MR2797562} this was dubbed {\em reverse factorization.})
Precursors of this idea have been well-known since the 1970s (cf. the
counter-example in~\cite{muckenhoupt-wheeden76}) but it was first
systematically developed (in the case $p=q$)
in~\cite[Chapter~6]{MR2797562}.   The following lemma was proved
in~\cite{MR3065302}.

\begin{lemma} \label{lemma:factored}
Given $0<\alpha<n$, suppose $1<p\leq q<\infty$ and
  $\frac{1}{p}-\frac{1}{q}\leq \frac{\alpha}{n}$ .  Let $w_1,\,w_2$ be
  locally integrable functions, and define
\[ u = w_1\big(M_\gamma w_2\big)^{-\frac{q}{p'}}, \qquad
\sigma = w_2\big(M_\gamma w_1\big)^{-\frac{p'}{q}}, \]
where
\[ \gamma = \frac{\frac{\alpha}{n}+\frac{1}{q}-\frac{1}{p}}
{\frac{1}{n}\left(1+\frac{1}{q}-\frac{1}{p}\right)}. \]
Then $(u,\sigma)\in A^\alpha_{p,q}$ and $[u,\sigma]_{A^\alpha_{p,q}}\leq 1$. 
\end{lemma}

\begin{proof}
By our assumptions on $p$, $q$ and $\alpha$, $0\leq \gamma \leq \alpha$.  
Fix a cube $Q$.  Then 
\begin{align*}
& |Q|^{\frac{\alpha}{n}+\frac{1}{q}-\frac{1}{p}}
\left(\avgint_Q w_1(x)(M_\gamma w_2(x)^{-\frac{q}{p'}}\,dx\right)^\frac{1}{q}
\left(\avgint w_2(x)\big(M_\gamma w_1(x)^{-\frac{p'}{q}} \,dx\right)^\frac{1}{p'} \\
& \qquad \qquad\leq  |Q|^{\frac{\alpha}{n}+\frac{1}{q}-\frac{1}{p}}
\left(\avgint w_1(x)\,dx\right)^\frac{1}{q'}
\left(|Q|^\frac{\gamma}{n} \left(\avgint_Q w_2(x)\,dx\right)\right)^{-\frac{1}{p'}} \\
& \qquad \qquad\qquad \qquad \times 
\left(\avgint w_2(x)\,dx\right)^\frac{1}{p'}
\left(|Q|^\frac{\gamma}{n} \left(\avgint_Q w_1(x)\,dx\right)\right)^{-\frac{1}{q}} \\
& \qquad \qquad=  |Q|^{\frac{\alpha}{n}+\frac{1}{q}-\frac{1}{p}-\frac{\gamma}{n}\left(1+\frac{1}{q}-\frac{1}{p}\right)} \\
&\qquad \qquad = 1. 
\end{align*}
\end{proof}

\begin{proof}[Construction of Example~\textup{\ref{example:bad-wt}}]
To construct the desired example, we need to consider two cases.  In
both cases we will work on the real line, so $n=1$.

Suppose first that $\frac{1}{p}-\frac{1}{q}>\alpha$.
Let $f=\sigma=\chi_{[-2,-1]}$ and let
$u=x^t\chi_{[0,\infty)}$, where $t=q(1-\alpha)-1$.  Given any
$Q=(a,b)$, $Q\cap \supp(u)\cap \supp(\sigma)=\emptyset$ unless $a<-1$ and $b>0$.  In
this case we have that
\begin{multline*}
|Q|^{\alpha +\frac{1}{q}-\frac{1}{p}}\|u
^{\frac1q}\|_{q,Q}\|\sigma^{\frac{1}{p'}}\|_{p',Q} \\
 \leq b^{\alpha+\frac1q-\frac1p}
\left(\frac1b\int_0^b x^t\,dx\right)^{\frac1q}
\left(\frac1b \int_{-2}^{-1} \,dx\right)^{\frac{1}{p'}} 
 \lesssim b^{\alpha + \frac{t+1}{q}-1} 
 = 1.
\end{multline*}
Hence, $(u,\sigma)\in A_{p,q}^\alpha$. 
On the other hand, for all $x>1$, 
\[ M_\alpha(f\sigma)(x) \approx x^{\alpha-1}, \]
and so
\[ \int_{\R} M_\alpha(f\sigma)(x)^qu(x)\,dx 
\gtrsim  \int_1^\infty x^{q(\alpha-1)}x^{q(1-\alpha)-1}\,dx
= \int_1^\infty \frac{dx}{x} = \infty. \]

\bigskip

Now suppose $\frac{1}{p}-\frac{1}{q}\leq \alpha$.  Fix $\gamma$ as in
Lemma~\ref{lemma:factored}.  We first construct a set $E\subset [0,\infty)$ such that
$M_\gamma(\chi_E)(x)\approx 1$ for $x>0$.  Let
\[ E = \bigcup_{j\geq 0} [j,j + (j+1)^{-\gamma}).  \]
Suppose $x\in [k,k+1)$; if  $k=0$, then it is immediate that if we take $Q=[0,2]$, then 
$M_\gamma(\chi_E) \geq 3\cdot 2^{\gamma-2}\approx 1$.  If 
$k\geq 1$, let $Q=[0,x]$; then 
\[ M_\gamma(\chi_E)(x) \geq 
x^{\gamma-1}\sum_{0\leq j \leq \lfloor x \rfloor} (j+1)^{-\gamma}
\geq (k+1)^{\gamma-1} \sum_{j=0}^{k} (j+1)^{-\gamma} \approx (k+1)^{\gamma-1} (k+1)^{1-\gamma} = 1. \]

\medskip

To prove the reverse inequality we will show that 
$|Q|^{\gamma-1}|Q\cap E|\lesssim 1$ for every cube $Q$.    If $|Q|\leq 1$, then 
\[ |Q|^{\gamma-1}|Q\cap E| \leq |Q|^{\gamma} \leq 1, \]
so we only have to consider $Q$ such that $|Q|\geq 1$.    In this
case, given  $Q$ let $Q'$ be the smallest interval whose endpoints
are integers that contains $Q$.  Then $|Q'|\leq |Q|+2 \leq 3|Q|$, and
so  $|Q|^{\gamma-1}|E\cap Q|\approx |Q'|^{\gamma-1}|E\cap Q'|$.
Therefore, without loss of generality, we may assume that
$Q=[a,a+h+1]$, where $a,\,h$ are non-negative integers. Then
\begin{multline*}
 |Q|^{\gamma-1}|Q\cap E| = (1+h)^{\gamma-1} \sum_{a\leq j \leq a+h} (j+1)^{-\gamma}
\approx (1+h)^{\gamma-1} \int_a^{a+h} (t+1)^{-\gamma}\,dt \\
\approx (1+h)^{\gamma-1}\big( (a+h+1)^{1-\gamma}-(a+1)^{1-\gamma}\big).
\end{multline*}
To estimate the last term suppose first that $h \leq a$.  Then by the mean value theorem the last term is dominated by 
\[  (1+h)^{\gamma-1}(1+h)(a+1)^{-\gamma}  \leq 1. \]
On the other hand, if $h>a$, then the last term is dominated by 
\[ (1+h)^{\gamma-1}(a+h+1)^{1-\gamma} \leq 2^{1-\gamma} \approx 1. \]
This completes the proof that $M_\gamma(\chi_E)(x)\approx 1$. 

\bigskip

We can now give our desired counter example.   Let $w_1=\chi_E$ and $w_2=\chi_{[0,1]}$.   Then for all $x\geq 2$, 
\[ M_\gamma w_1(x) \approx 1, \qquad 
M_\gamma w_2(x) =\sup_{Q} |Q|^{\gamma-1}\int_Q w_2(y)\,dy \approx x^{\gamma-1}. \]
Define
\[ u = w_1 (M_\gamma w_2)^{-\frac{q}{p'}}, \qquad
\sigma = w_2 (M_\gamma w_1)^{-\frac{p'}{q}}; \]
then by Lemma~\ref{lemma:factored}, $(u,\sigma)\in A^\alpha_{p,q}$.  Moreover, for $x\geq 2$, we have that 
\[ u(x) \approx x^{(1-\gamma)\frac{q}{p'}}\chi_E(x), \qquad  \sigma(x) \approx \chi_{[0,1]}(x). \]

Fix $f\in L^p(\sigma)$: without loss of generality, we may assume $\supp(f)\subset [0,1]$.  Then $f\sigma$ is locally integrable, and for $x\geq 2$ we have that
\[ M_\alpha (f\sigma)(x) \geq x^{\alpha -1} \|f\sigma\|_1 \approx x^{\alpha-1}. \]
Therefore, for $x\geq 2$, 
\[ M_\alpha(f\sigma)(x)^q u(x) \gtrsim x^{(\alpha-1)q}x^{(1-\gamma) \frac{q}{p'}}\chi_E(x). \]
By the definition of $\gamma$, 
\[ \gamma\left(\frac{1}{q}+\frac{1}{p'}\right) = 
\gamma\left(1+\frac{1}{q}-\frac{1}{p}\right) = \alpha+\frac{1}{q}-\frac{1}{p}
= \alpha-1 + \frac{1}{q}+\frac{1}{p'}; \]
equivalently,
\[ (\gamma-1)\left(\frac{q}{p'}+1\right) = q(\alpha-1),\]
and so
\[ (\alpha-1)q+(1-\gamma)\frac{q}{p'} = \gamma-1. \]

Therefore, to show that $M_\alpha(f\sigma)\not\in L^q(u)$, it will be
enough to prove that
\[ \int_2^\infty x^{\gamma-1}\chi_E(x)\,dx =\infty, \]
but this is straightforward:
\begin{multline*}
\int_2^\infty x^{\gamma-1}\chi_E(x)\,dx
 = \sum_{j=2}^\infty \int_j^{j+(j+1)^{-\gamma}} x^{\gamma-1}\,dx 
 \geq \sum_{j=2}^\infty (j+(j+1)^{-\gamma})^{\gamma-1} (j+1)^{-\gamma} \\
 \geq \sum_{j=2}^\infty (j+1)^{\gamma-1} (j+1)^{-\gamma} 
 \geq \sum_{j=2}^\infty  (j+1)^{-1} 
 = \infty. 
\end{multline*}
\end{proof}

\medskip

If we combine Example~\ref{example:bad-wt} with the pointwise inequalities in
Section~\ref{section:dyadic}, we see that the $A_{p,q}^\alpha$ condition is
also not sufficient for
the fractional integral operator to satisfy the strong type
inequality.    This condition is also not sufficient for the weak
$(p,q)$ inequality.    A counter-example when $p=q=n=2$ and
$\alpha=\frac{1}{2}$ using measures was constructed by Kerman and
Sawyer~\cite{MR867921}.   Here we construct a general counter-example
that holds for all $p$, $q$ and $\alpha$.  For simplicity we construct
the example for $n=1$, but it can be modified to work in all
dimensions.  We want to thank E.~Sawyer for useful comments on an
earlier version of this construction.

\begin{example} \label{example:bad-wt-frac}
Let $n=1$.  Given $1<p\leq q < \infty$ and $0<\alpha<1$, 
there exists a pair of weights $(u,\sigma)\in A^\alpha_{p,q}$ and a
non-negative function $f\in L^p(\sigma)$ such that 
\[ \sup_{t>0} t\, u(\{ x\in \R :  I_\alpha(f\sigma)(x)>t \})^{\frac{1}{q}}  =
\infty. \]
\end{example}

\begin{proof}
  Fix $p$, $q$ and $\alpha$ and let $u=\chi_{[-1,1]}$.  We will first
  construct a non-negative weight $\sigma$ such that
  $[u,\sigma]_{A_{p,q}^\alpha}<\infty$.  We will then
  find a non-negative function $f\in L^p(\sigma)$ such that
  $I_\alpha(f\sigma)(x)=\infty$ for all $x\in (0,1)$. 
  Then we have that 
\[ \sup_{t>0} t\, u(\{ x\in \R^n : I_\alpha (f\sigma)(x)>t \})^{\frac{1}{q}}
\geq  \sup_{t>0} t\, u([0,1])^{\frac{1}{q}} = \infty. \]

\medskip

Let $\sigma = |x|^{-r}\chi_{\{|x|>1\}}$, where $r$ is defined by
\[ \alpha - \frac{1}{p} = \frac{r}{p'}. \]
Given that $u$ and $\sigma$ are symmetric around the origin and have
disjoint supports, it is immediate that  to check the $A_{p,q}^\alpha$
condition it suffices to check it on intervals $Q=[0,t]$, $t>1$.  But
in this case,
\[ |Q|^{\alpha+\frac{1}{q}-\frac{1}{p}}
\left(\avgint_Q u(x)\,dx\right)^{\frac{1}{q}}\left(\avgint_Q \sigma(x)\,dx\right)^{\frac{1}{p'}}
= t^{\alpha+\frac{1}{q}-\frac{1}{p}} t^{-\frac{1}{q}} 
\left(\frac{1}{t}\int_1^t x^{-r}\,dx\right)^{\frac{1}{p'}}. \]
If $r<1$ then $x^{-r}$ is locally integrable at the origin, and so by
our choice of $r$, the right hand term is bounded by 
\[  t^{\alpha+\frac{1}{q}-\frac{1}{p}} t^{-\frac{1}{q}}
\left(\frac{1}{t}\int_0^t x^{-r}\,dx\right)^{\frac{1}{p'}} 
\approx  t^{\alpha+\frac{1}{q}-\frac{1}{p}-\frac{1}{q}-\frac{r}{p'}} =1. \]
On the other hand, if $r>1$, then $x^{-r} \in L^1(\R)$, and so the
right hand side is bounded by
\[ t^{\alpha+\frac{1}{q}-\frac{1}{p}-\frac{1}{q}-\frac{1}{p'}} = t^{\alpha-1} \leq 1. \]
Hence, $[u,\sigma]_{A_{p,q}^\alpha}<\infty$.

\medskip

We now construct $f$ with the desired properties.  Let
\[ f(x) = \frac{x^{r-\alpha}}{\log(ex)} \chi_{(1,\infty)}(x);  \]
then
\[ f(x)^p\sigma(x) = \frac{x^{(r-\alpha)p-r}}{\log(ex)^p} \chi_{(1,\infty)}(x). \]
By our definition of $r$,
\[ \alpha - \frac{1}{p} = r\left(1-\frac{1}{p}\right), \]
or equivalently,
\[ \frac{r}{p} - \frac{1}{p} = r - \alpha, \]
which in turn implies that $p(r-\alpha)=r-1$.  Hence, since $p>1$,
\[ f(x)^p\sigma(x) = \frac{1}{x\log(ex)^p} \chi_{(1,\infty)}(x) \in L^1(\R). \]

On the other hand, for $x\in (0,1)$, 
\begin{multline*}
 I_\alpha(f\sigma)(x) = \int_1^\infty
 \frac{f(y)\sigma(y)}{|x-y|^{1-\alpha}}\,dy \\
= \int_1^\infty \frac{dy}{y^{\alpha}(y-x)^{1-\alpha}\log(ey)} 
\geq \int_1^\infty \frac{dy}{y\log(ey)}   = +\infty.  
\end{multline*}
This completes the proof.
\end{proof}

Though it does not matter for our proof, we note in passing that in
this example we actually have that $I_\alpha(f\sigma)(x)=\infty$
for all $x$.

\subsection*{Young functions and Orlicz norms}

Given the failure of the $A_{p,q}^\alpha$ condition to be sufficient
for strong type norm inequalities for fractional maximal and integral
operators, our goal is to generalize this condition to get one that
is sufficient, resembles the $A_{p,q}^\alpha$ condition and
shares its key properties.  In particular, the condition should be
``geometric'' in the sense that, unlike the testing conditions in
Section~\ref{section:testing}, it does not involve the operator
itself, and it should interact well with dyadic grids.  Our approach
will be to replace the $L^q$ and $L^{p'}$ norms in the definition with
larger norms.  For $A_p$ weights this was first done by Neugebauer,
who replaced the $L^p$ and $L^{p'}$ norms with $L^{rp}$ and $L^{rp'}$
norms, $r>1$.  P\'erez~\cite{MR1291534,MR1327936} greatly extended
this idea by showing that Orlicz norms that lie between $L^p$ and
$L^{rp}$ for any $r>1$ will also work.

To formulate his approach we first need to introduce some basic ideas about
Young functions and Orlicz norms.  For complete information
see~\cite{krasnoselskii-rutickii61,rao-ren91}. 
A function $B : [0,\infty)\rightarrow[0,\infty)$ is a Young function
if it is continuous, convex and strictly increasing, if $B(0)=0$, and if
$B(t)/t\rightarrow\infty$ as $t\rightarrow\infty$.  $B(t)=t$ is not
properly a Young function, but in many instances what we say applies
to this function as well.  It is convenient, particularly when
computing constants, to assume $B(1)=1$, but this normalization is not
necessary.   A Young function $B$ is said to be doubling if there exists a positive
constant $C$ such that $B(2t)\leq CB(t)$ for all $t>0$.

Given a Young function $B$ and  a cube $Q$, we define the normalized Luxemburg norm of $f$ 
on $Q$  by
\begin{equation}
\|f\|_{B,Q} = \inf\left\{ \lambda > 0 :
\avgint_Q B\left(\frac{|f(x)|}{\lambda}\right)\,dx \leq 1
\right\}.
\label{lux-norm}
\end{equation}
When $B(t)=t^p$, $1\leq p<\infty$, the Luxemburg norm coincides with the normalized $L^p$ norm:
\[ \|f\|_{B,Q} = \left(\avgint_Q |f(x)|^p\,dx\right)^{1/p}
= \|f\|_{p,Q}. \]

If $A(t) \leq B(ct)$ for all $t\geq t_0>0$,  then there exists a constant $C$,
depending only on $A$ and $B$, such that for all cubes $Q$ and functions
$f$, $\|f\|_{A,Q} \leq C\|f\|_{B,Q}$.

Given a Young function $B$, the associate Young
function $\bar{B}$ is defined by
\[ \bar{B}(t) = \sup_{s>0} \{ st - B(s) \},\qquad t>0; \]
$B$ and $\bar{B}$ satisfy 
\[ t\leq B^{-1}(t)\bar{B}^{-1}(t) \leq 2t. \]
Note that the associate of $\bar{B}$ is again $B$.  Using the
associate Young function,
H\"older's inequality can be generalized to the scale of Orlicz
spaces:  given any Young function $B$, then for all functions $f$ and $g$ and all
cubes $Q$,
\[ 
 \avgint_Q |f(x)g(x)|\,dx \leq 2\|f\|_{B,Q}\|g\|_{\bar{B},Q}.
\]
More generally, if $A$, $B$ and $C$ are Young functions such that
for all $t\geq t_0>0$,
\[ B^{-1}(t)C^{-1}(t) \leq  cA^{-1}(t), \]
then
\[  \|fg\|_{A,Q} \leq K \|f\|_{B,Q}\|g\|_{C,Q}. \]

Below we will need to impose a growth condition on Young functions
that compares them to powers of $t$.  This condition was first
introduced by P\'erez~\cite{MR1327936}.  Given $1<p<\infty$, we say that
a Young function $B$ satisfies the $B_p$ condition if
\[ \int_1^\infty \frac{B(t)}{t^p}\frac{dt}{t} < \infty. \]

Frequently, we will want to make an assumption of the form $\bar{B}
\in B_p$.  If both $B$ and $\bar{B}$ are doubling, then this is
equivalent to 
\begin{equation} \label{eqn:alt-Bp}
  \int_1^\infty
 \left(\frac{t^{p'}}{B(t)}\right)^{p-1}\frac{dt}{t} < \infty.
\end{equation}
(See~\cite[Proposition~5.10]{MR2797562}.)
There are two important examples of functions that satisfy the $B_p$ condition.
If $B(t)=t^{rp'}$, $r>1$,  or if
$B(t)=t^p\log(e+t)^{p-1+\delta}$, $\delta>0$, 
then $\bar{B} \in B_p$.  For reasons that will be clear below, we will
refer to these as power bumps and log bumps.   One essential property
of this condition is that if $\bar{B}\in B_p$, then $\bar{B}\lesssim
t^p$ and $B \gtrsim t^{p'}$.  Note in particular that if
$B(t)=t^{p'}$, then $\bar{B}(t)=t^p$ is not in $B_p$.  

The $B_p$ condition was introduced by P\'erez to characterize the
boundedness of the Orlicz maximal operator.  Given a Young function
$B$ and a measurable function $f$, define
\[ M_B f(x) = \sup_Q \|f\|_{B,Q}\,\chi_Q(x). \]

\begin{prop} \label{prop:perez-Bp}
Given a Young function $B$ and $1<p<\infty$, the following are
equivalent:
\begin{enumerate}
\item $B\in B_p$;

\item for all $f\in L^p$,
\[ \left(\int_\subRn M_Bf(x)^p\,dx\right)^{\frac{1}{p}}
\leq C(n,p)\left(\int_\subRn |f(x)|^p\,dx\right)^{\frac{1}{p}}. \]
\end{enumerate}
\end{prop}

As given in~\cite{MR1327936}, Proposition~\ref{prop:perez-Bp} included
the assumption that $B$ was doubling.  However, this assumption was
only included to use the $B_p$ condition in the form
of~\eqref{eqn:alt-Bp} to prove sufficiency.  This was correctly noted
in~\cite{MR2797562}, but we made the incorrect assertion that it was
not needed for the proof of necessity in~\cite{MR1327936}.  However, Liu and
Luque~\cite{MR3256181} recently gave a proof that it is necessary without
assuming doubling.  

We can also define a fractional Orlicz maximal
operator $M_{B,\alpha}$:  see Section~\ref{section:separated}
below.  There is also a corresponding $B_{p}^\alpha$ condition which
is useful in determining sharp
constants estimates for the fractional integral operator:
see~\cite{MR3065302} for details.

\subsection*{The $A_{p,q}^\alpha$ bump conditions}

Using the machinery introduced above, we can now state our
generalizations of the $A_{p,q}^\alpha$ condition.   Given
$0<\alpha<n$, $1<p\leq q<\infty$, Young functions $A$ and $B$,
$\bar{A}\in B_{q'}$ and $\bar{B} \in B_p$, and a pair of weights
$(u,\sigma)$, we define 
\begin{gather*} 
[u,\sigma]_{A_{p,q,B}^\alpha} 
= \sup_Q |Q|^{\frac{\alpha}{n}+\frac{1}{q}-\frac{1}{p}}
\|u^\frac{1}{q}\|_{q,Q}\|\sigma^{\frac{1}{p'}}\|_{B,Q} < \infty, \\
[u,\sigma]_{A_{p,q,A}^\alpha}^* 
= \sup_Q |Q|^{\frac{\alpha}{n}+\frac{1}{q}-\frac{1}{p}}
\|u^\frac{1}{q}\|_{A,Q}\|\sigma^{\frac{1}{p'}}\|_{p',Q} < \infty. 
\end{gather*}
By our hypotheses on $A$ and $B$ both of these quantities
are larger than $[u,\sigma]_{A_{p,q}^\alpha}$:  we have ``bumped up''
one of the norms in the scale of Orlicz spaces.  For this reason we
refer to these as $A_{p,q}^\alpha$ bump conditions.  

Note that the second
condition is the ``dual'' of the first, in the sense that
\[ [u,\sigma]_{A_{p,q,A}^\alpha}^* =
[\sigma,u]_{A_{q',p',A}^\alpha}. \]
As we will see below, this condition will play a role analogous to
that of the dual testing conditions discussed in
Section~\ref{section:testing}.    Informally, it is common to refer to
the $[u,\sigma]_{A_{p,q,B}^\alpha}$ condition as having a bump on
the right, and the $[u,\sigma]_{A_{p,q,A}^\alpha}^*$ as having a bump
on the left, and collectively we refer to these as separated bump
conditions. 

We can combine these conditions by putting a bump on both norms
simultaneously:
\[ [u,\sigma]_{A_{p,q,A,B}^\alpha} 
= \sup_Q |Q|^{\frac{\alpha}{n}+\frac{1}{q}-\frac{1}{p}}
\|u^\frac{1}{q}\|_{A,Q}\|\sigma^{\frac{1}{p'}}\|_{B,Q} < \infty. \]
We refer this as a conjoined bump condition.  Clearly, it is larger
than either of the separated bump conditions.  In fact, assuming the
conjoined bump condition is stronger than assuming both separated bump
conditions.    The following example (with $\alpha=0$, $p=2$) was constructed
in~\cite{anderson-DCU-moen}.

\begin{example} \label{example:sep-weaker}
Given
$0\leq\alpha<n$ and $1<p\leq q<\infty$, there exists a pair of Young functions $A$ and $B$,
$\bar{A}\in B_{q'}$ and $\bar{B} \in B_p$, and a pair of weights
$(u,\sigma)$, such that $[u,\sigma]_{A_{p,q,B}^\alpha}, \,
[u,\sigma]_{A_{p,q,A}^\alpha}^*<\infty$, but
$[u,\sigma]_{A_{p,q,A,B}^\alpha} =\infty$.
\end{example}

\begin{proof}
We construct our example on the real line, so $0<\alpha<1$.  Define the Young functions
\[ A(t)=  t^q \log(e+t)^q, \quad B(t)= t^{p'} \log(e+t)^{p'}.  \]
Then $\bar{A} \in B_{q'}$ and $\bar{B} \in B_p$. By rescaling, if we let
$\Psi(t)=t\log(e+t)^{q}, \, \Phi(t)=t\log(e+t)^{p'}$, then for any pair $(u,\sigma)$, 
\[ \|u^{\frac{1}{q}}\|_{A,Q} \approx \|u\|_{\Psi,Q}^{\frac{1}{q}}, \qquad
\|\sigma^{\frac{1}{p'}}\|_{B,Q} \approx \|\sigma\|_{\Phi,Q}^{\frac{1}{p'}}. \]
Therefore, it will suffice  to estimate the norms of $u$ and $\sigma$ with respect
to $\Psi$ and $\Phi$.   Similarly, we can replace the localized $L^q$
and $L^{p'}$ norms of
$u^{\frac{1}{q}}$ and $\sigma^{\frac{1}{p'}}$ with the $L^1$ norms of $u$ and $\sigma$.

Before we define $u$ and $\sigma$ we first construct a pair
$(u_0,\sigma_0)$ which will be the basic building block for our
example.  Fix an integer $k\geq 2$ and  define
$Q=(0,k)$, $\sigma_0=\chi_{(0,1)}$ and   $u_0 = K_k^q\chi_{(k-1,k)}$,
where $K_k=k^{1-\alpha}\log(e+k)^{-\frac{3}{2}}$.
Since $\Psi^{-1}(t)\approx t\log(e+t)^{-q}$, $\Phi^{-1}(t)\approx
t\log(e+t)^{-p'}$, by 
the definition of the Luxemburg norm,
\[ \|u_0\|_{1,Q}^{\frac{1}{q}} = \frac{K_k}{k^{\frac{1}{q}}}, \;
\|u_0\|_{\Psi,Q} \approx \frac{K_k\log(e+k)}{k^{\frac{1}{q}}},  \qquad 
\|\sigma_0\|_{1,Q}^{\frac{1}{p'}} = \frac{1}{k^{\frac{1}{p'}}}, \; 
\|\sigma_0\|_{\Phi,Q} \approx\frac{\log(e+k)}{k^{\frac{1}{p'}}}. \]
Therefore, we have that 
\[  |Q|^{\alpha+\frac{1}{q}-\frac{1}{p}}\|u_0\|_{1,Q}^{\frac{1}{q}}\|\sigma_0\|_{\Phi,Q}^{\frac{1}{p'}}, \quad
|Q|^{\alpha+\frac{1}{q}-\frac{1}{p}}
\|u_0\|_{\Phi,Q}^{\frac{1}{q}}\|\sigma_0\|_{1,Q}^{\frac{1}{p'}} \approx \frac{1}{\log(e+k)^{\frac{1}{2}}}, \]
but
\[ |Q|^{\alpha+\frac{1}{q}-\frac{1}{p}}
\|u_0\|_{\Phi,Q}^{\frac{1}{q}}\|\sigma_0\|_{\Phi,Q}^{\frac{1}{p'}} \approx \log(e+k)^{\frac{1}{2}}. \]

\medskip

We now define $u$ and $\sigma$ as follows:
\[ 
u(x) = \sum_{k\geq 2} K_k^q\chi_{I_k}(x), \qquad
\sigma(x) = \sum_{k\geq 2} \chi_{J_k}(x).
\]
where $I_k=(e^k+k-1,e^k+k)$ and $J_k= (e^k,e^k+1)$.
Since the above computations are translation invariant, we immediately
get that if $Q_k=(e^k,e^k+k)$, then 
\[ |Q_k|^{\alpha+\frac{1}{q}-\frac{1}{p}}
\|u\|_{\Phi,Q_k}^{\frac{1}{q}}\|\sigma\|_{\Phi,Q_k}^{\frac{1}{p'}} \approx \log(e+k)^{\frac{1}{2}}, \]
and so $[u,\sigma]_{A_{p,q,A,B}^\alpha}=\infty$.   

We will now prove
that $[u,\sigma]_{A_{p,q,B}^\alpha}$ and $[\sigma,u]_{A_{p,q,A}^\alpha}^*$ are both finite.  We
will show $[u,\sigma]_{A_{p,q,A}^\alpha}^*<\infty$; the argument for the
first condition is
essentially the same.    Fix an interval $Q$; we will show that
$|Q|^{\alpha+\frac{1}{q}-\frac{1}{p}}\|u\|_{\Psi,Q}^{\frac{1}{q}}\|\sigma\|_{1,Q}^{\frac{1}{p'}}$
is uniformly bounded.   Let $N$ be an integer such that 
$N-1 \leq |Q| \leq N$.   We need to consider those values of $k$ such
that $Q$ intersects either $I_k$ or $J_k$.  

Suppose that for some $k\geq N+2$, $Q$ intersects $I_k$.  But in this
case it cannot intersect $J_j$ for any $j$ and so
$\|\sigma\|_{1,Q}=0$.  Similarly, if $Q$ intersects $J_k$, then
$\|u\|_{\Psi,Q}=0$. 

Now suppose that for some $k<N+2$, $Q$ intersects one of $I_k$ or
$J_k$.  If $\log(N) \lesssim k$ (more precisely, if
$N<e^k-e^{k-1}-1$), then for any $j\neq k$, $Q$ cannot intersect $I_j$
or $J_j$.  In this case 
$|Q|^{\alpha+\frac{1}{q}-\frac{1}{p}}\|u\|_{\Psi,Q}^{\frac{1}{q}}\|\sigma\|_{1,Q}^{\frac{1}{p'}}\neq 0$ only if $Q$ intersects both
$I_k$ and $J_k$, and will reach its maximum when $N\approx k$.  But in
this case we can 
replace $Q$ by $(e^k,e^k+k)$ and the above computation shows
that $|Q|^{\alpha+\frac{1}{q}-\frac{1}{p}}\|u\|_{\Psi,Q}^{\frac{1}{q}}\|\sigma\|_{1,Q}^{\frac{1}{p'}}\lesssim 1$.

Finally, suppose $Q$ intersects one or more pairs $I_k$ and $J_k$ with
$k \lesssim \log(N)$.   Then $|\supp(u)\cap Q|\lesssim \log(N)$ and 
$\|u\|_{L^\infty(Q)} \approx K_{\lfloor\log(N)\rfloor}^q \lesssim
\log(N)^{q(1-\alpha)}$.  Therefore, for any $r>1$,
\[ \|u\|_{\Psi,Q}^{\frac{1}{q}} \lesssim \|u\|_{r,Q}^{\frac{1}{q}} \leq 
\|u\|_{L^\infty(Q)}^{\frac{1}{q}} \left(\frac{|\supp(u)\cap Q|}{|Q|}\right)^{\frac{1}{rq}}
\lesssim \frac{\log(N)^{1-\alpha+\frac{1}{rq}}}{N^{\frac{1}{rq}}}. \]
A similar calculation shows that
\[ \|\sigma \|_{1,Q}^{\frac{1}{p'}} \lesssim \left(\frac{\log(N)}{N}\right)^{\frac{1}{p'}}. \]
Hence, we have that  
\[
|Q|^{\alpha+\frac{1}{q}-\frac{1}{p}}\|u\|_{\Phi,Q}^{\frac{1}{q}}\|\sigma\|_{1,Q}^{\frac{1}{p'}}
\lesssim N^{\alpha+\frac{1}{q}-\frac{1}{p}-\frac{1}{rq}-\frac{1}{p'}}
\log(N)^{1-\alpha+\frac{1}{rq}+\frac{1}{p'}}. \]
Since $\alpha<1$, if we fix $r>1$ sufficiently close to 1 we have that
the exponent on $N$ is negative, and so this quantity will be
uniformly bounded for all $N$.  
We thus have that $[u,\sigma]_{A_{p,q,A}^\alpha}^*<\infty$ and our
proof is complete.
\end{proof}

\subsection*{Bump conditions for fractional maximal operators}

There is a parallel between bump conditions and the testing conditions
described in Section~\ref{section:testing}.  For maximal operators,
only a single testing condition is needed for the strong type
inequality; similarly, only a single bump (on the right) is required
to get a sufficient condition.  The following result is due to
P\'erez~\cite{MR1291534,MR1327936} and our proof is based on his.

\begin{theorem} \label{thm:strong-max-bump}
Given $0\leq \alpha<n$, $1<p\leq q <\infty$, and a Young function $B$
such that $\bar{B}\in B_p$, suppose the pair of
weights $(u,\sigma)$ is such that
$[u,\sigma]_{A_{p,q,B}^\alpha}<\infty$.
Then for every $f\in L^p(\sigma)$,
\[ \left(\int_\subRn M_\alpha (f\sigma)(x)^q u(x)\,dx\right)^{\frac{1}{q}}
\leq C(n,p,q) [u,\sigma]_{A_{p,q,B}^\alpha}\left(\int_\subRn
  |f(x)|^p \sigma(x)\,dx\right)^{\frac{1}{p}}. \]
\end{theorem}

\medskip

Note that while our proof shows directly that the constant
depends linearly on $[u,\sigma]_{A_{p,q,B}^\alpha}$, in fact this is
always true in two weight inequalities.  This is an observation due to
Sawyer:  see~\cite[Remark~1.4]{cruz-moen2012}.

\begin{proof}
Arguing as we did in the proof of
Theorem~\ref{thm:strong-max-one}, we may assume that $f$ is
non-negative, bounded
and has compact support, and it will suffice to prove the desired
inequality for $L_\alpha^\Ss$, where $\Ss$ is a sparse subset of a
dyadic grid $\D$.   Indeed, we begin as we did
there, using the fact that the sets $E(Q)$ are disjoint.  But 
instead of the $A_\infty$ property we will use the generalized H\"older's
inequality to introduce the Orlicz maximal operator.  This allows us
to sum over the cubes in $\Ss$ and apply
Proposition~\ref{prop:perez-Bp} to get the desired estimate:
\begin{align*}
\|L_\alpha^\Ss (f\sigma)\|_{L^q(u)}^q 
& = \sum_{Q\in \Ss} |Q|^{q\frac{\alpha}{n}}\avg{f\sigma}_Q^q u(E(Q)) \\
& =\sum_{Q\in \Ss}|Q|^{q\frac{\alpha}{n}+1-\frac{q}{p}}\avg{f\sigma}_Q^q
\avg{u}_Q |Q|^{\frac{q}{p}} \\
& \leq 2 ^{\frac{q}{p}+1}\sum_{Q\in \Ss}|Q|^{q\frac{\alpha}{n}+1-\frac{q}{p}}
\avg{u}_Q \|\sigma^{\frac{1}{p'}}\|_{B,Q}^q
  \|f\sigma^{\frac{1}{p}}\|_{\bar{B},Q}^q |E(Q)|^{\frac{q}{p}} \\
& \leq 2 ^{\frac{q}{p}+1}[u,\sigma]_{A_{p,q,B}^\alpha}^q\sum_{Q\in \Ss}
 \|f\sigma^{\frac{1}{p}}\|_{\bar{B},Q}^q |E(Q)|^{\frac{q}{p}} \\
& \leq 2 ^{\frac{q}{p}+1}[u,\sigma]_{A_{p,q,B}^\alpha}^q\left(\sum_{Q\in \Ss}
 \|f\sigma^{\frac{1}{p}}\|_{\bar{B},Q}^p |E(Q)|\right)^{\frac{q}{p}}
\\
& \leq 2 ^{\frac{q}{p}+1}[u,\sigma]_{A_{p,q,B}^\alpha}^q\left(\sum_{Q\in\Ss}
\int_{E(Q)} M_{\bar{B}}(f\sigma^{\frac{1}{p}})(x)^p\,dx \right)^{\frac{q}{p}} \\
& \leq 2 ^{\frac{q}{p}+1}[u,\sigma]_{A_{p,q,B}^\alpha}^q\left(
\int_\subRn M_{\bar{B}}(f\sigma^{\frac{1}{p}})(x)^p\,dx
\right)^{\frac{q}{p}} \\
& \leq C(n,p,q) [u,\sigma]_{A_{p,q,B}^\alpha}^q\left(
\int_\subRn |f(x)|^p \sigma(x)\,dx\right)^{\frac{q}{p}}. \\
\end{align*}
\end{proof}

\subsection*{Bump conditions for fractional integral operators}
We now consider bump conditions for the fractional
integral operator.   For the strong type condition, we need two bumps,
analogous to the fact that you need two testing conditions.   Our
first result is for conjoined bumps; we will discuss separated bump
conditions in Section~\ref{section:separated} below.  Theorem~\ref{thm:frac-joined} was originally proved by
P\'erez~\cite{MR1291534} and our proof is modeled on his.

\begin{theorem} \label{thm:frac-joined}
Given $0< \alpha<n$, $1<p\leq q <\infty$, and Young functions $A,\,B$
such that $\bar{A}\in B_{q'}$ and $\bar{B}\in B_p$, suppose the pair of
weights $(u,\sigma)$ is such that
$[u,\sigma]_{A_{p,q,A,B}^\alpha}<\infty$.
Then for every $f\in L^p(\sigma)$,
\[ \left(\int_\subRn |I_\alpha (f\sigma)(x)|^q u(x)\,dx\right)^{\frac{1}{q}}
\leq C(n,p) [u,\sigma]_{A_{p,q,A,B}^\alpha}\left(\int_\subRn
  |f(x)|^p \sigma(x)\,dx\right)^{\frac{1}{p}}. \]
\end{theorem}

We note that Theorem~\ref{thm:frac-joined} was very influential in the
study of two weight norm inequalities, and it led to the conjecture that an
analogous result held for singular integral operators.    This problem
was solved recently by Lerner~\cite{Lern2012};  for prior results
see~\cite{cruz-uribe-martell-perezP,dcu-martell-perez,cruz-uribe-perez02}. 

\begin{proof}
Arguing as we did in the proof of Theorem~\ref{thm:strong-frac-one}, we may assume that $f$ is
non-negative, bounded and has compact support.  Further, it will suffice to prove the desired
inequality for $I_\alpha^\Ss$, where 
$\Ss$ is a sparse subset of a dyadic grid $\D$.

We  begin as in the one weight case by applying duality.  But here we
use the generalized H\"older's inequality to introduce two Orlicz
maximal operators and use these to sum over cubes in $\Ss$.  We can
then apply Proposition~\ref{prop:perez-Bp} twice.  More precisely,
by duality there exists $g\in L^{q'}(u)$, $\|g\|_L^{q'}(u)=1$, such that
\begin{align*}
& \|I_\alpha^\Ss(f\sigma)\|_{L^{q'}(u)} \\
& \qquad = \int_\subRn I_\alpha^\Ss(f\sigma) g(x)u(x)\,dx \\
& \qquad = \sum_{Q\in \Ss} |Q|^{\frac{\alpha}{n}}\avg{f\sigma}_Q \avg{gu}_Q |Q| \\
& \qquad \leq 2^2\sum_{Q\in \Ss}  |Q|^{\frac{\alpha}{n}+\frac{1}{q}-\frac{1}{p}}
\|u^{\frac{1}{q}}\|_{A,Q} \|\sigma^{\frac{1}{p'}}\|_{B,Q}
\|f\sigma^{\frac{1}{p}}\|_{\bar{B},Q}
\|gu^{\frac{1}{q'}}\|_{\bar{A},Q} |Q|^{\frac{1}{p}+\frac{1}{q'}} \\
& \qquad\leq 2^{\frac{1}{p}+\frac{1}{q'}+2}[u,\sigma]_{A_{p,q,A,B}^\alpha}
\sum_{Q\in \Ss} \|f\sigma^{\frac{1}{p}}\|_{\bar{B},Q}
|E(Q)|^{\frac{1}{p}}
\|gu^{\frac{1}{q'}}\|_{\bar{A},Q} |E(Q)|^{\frac{1}{q'}}. \\
\intertext{By H\"older's inequality and the fact that $q'\leq p'$, we
  have that}
& \qquad\leq 16[u,\sigma]_{A_{p,q,A,B}^\alpha}
\left(\sum_{Q\in \Ss} \|f\sigma^{\frac{1}{p}}\|_{\bar{B},Q}^p
|E(Q)|\right) ^{\frac{1}{p}}
\left(\sum_{Q\in \Ss} \|gu^{\frac{1}{q'}}\|_{\bar{A},Q}^{p'}
  |E(Q)|^{\frac{p'}{q'}}\right)^{\frac{1}{p'}} \\
& \qquad\leq 16[u,\sigma]_{A_{p,q,A,B}^\alpha}
\left(\sum_{Q\in \Ss} \|f\sigma^{\frac{1}{p}}\|_{\bar{B},Q}^p
|E(Q)|\right) ^{\frac{1}{p}}
\left(\sum_{Q\in \Ss} \|gu^{\frac{1}{q'}}\|_{\bar{A},Q}^{q'}
  |E(Q)|\right)^{\frac{1}{q'}} \\
& \qquad\leq 16[u,\sigma]_{A_{p,q,A,B}^\alpha}
\left(\sum_{Q\in \Ss} \int_{E(Q)}
  M_{\bar{B}}(f\sigma^{\frac{1}{p}})(x)^p\,dx \right) ^{\frac{1}{p}}
\left(\sum_{Q\in \Ss} \int_{E(Q)}
  M_{\bar{A}}(gu^{\frac{1}{q'}})(x)^{q'}\,dx \right) ^{\frac{1}{q'}}\\
& \qquad\leq 16[u,\sigma]_{A_{p,q,A,B}^\alpha} 
\left(\int_\subRn   M_{\bar{B}}(f\sigma^{\frac{1}{p}})(x)^p\,dx
\right) ^{\frac{1}{p}}
\left(\int_\subRn  M_{\bar{A}}(gu^{\frac{1}{q'}})(x)^{q'}\,dx \right)
^{\frac{1}{q'}}\\
& \qquad \leq C(n,p,q) [u,\sigma]_{A_{p,q,A,B}^\alpha} 
\left(\int_\subRn  f(x)^p\sigma(x)\,dx
\right) ^{\frac{1}{p}}.
\end{align*}
\end{proof}

\subsection*{Bump conditions for commutators}
Finally, we prove a conjoined bump condition for commutators.  Because
commutators are more singular, we need stronger bump conditions.
To use the fact that $b$ is a $BMO$ function, it is most natural to
state these in terms of log bumps.  This result was originally proved
in~\cite{cruz-moen2012}; our proof is a simplification of the argument
given there.

\begin{theorem} \label{thm:comm-joined}
Given $0< \alpha<n$, $1<p\leq q <\infty$, and $b\in BMO$, suppose the pair of
weights $(u,\sigma)$ is such that
$[u,\sigma]_{A_{p,q,A,B}^\alpha}<\infty$, where
\[ A(t)=t^q\log(e+t)^{2q-1+\delta}, \quad
B(t)=t^{p'}\log(e+t)^{2p'-1+\delta}, \quad \delta >0. \]
Then for every $f\in L^p(\sigma)$,
\[ \left(\int_\subRn |[b,I_\alpha] (f\sigma)(x)|^q u(x)\,dx\right)^{\frac{1}{q}}
\leq C(n,p)\|b\|_{BMO} [u,\sigma]_{A_{p,q,A,B}^\alpha}\left(\int_\subRn
  |f(x)|^p \sigma(x)\,dx\right)^{\frac{1}{p}}. \]
\end{theorem}

The proof requires one lemma, which generalizes a result due to Sawyer
and Wheeden~\cite[p.~829]{MR1175693} and is proved in much the same
way.

\begin{lemma} \label{lemma:sum-alpha}
Fix $0<\alpha<n$,  a dyadic grid $\D$, and a Young function $\Phi$.  Then for any $P\in
\D$ and any function $f$,
\[ \sum_{\substack{Q\in \D \\ Q\subset P}} 
|Q|^{\frac{\alpha}{n}} |Q|\|f\|_{\Phi,Q}
\leq C(\alpha) |P|^{\frac{\alpha}{n}} |P|\|f\|_{\Phi,P}. \]
\end{lemma}

\begin{proof}
To prove this we need to replace the  Luxemburg norm with the equivalent Amemiya norm~\cite[Section~3.3]{rao-ren91}:  
\[ \|f\|_{\Phi,P} \leq  \inf_{\lambda>0} \left\{ \lambda
  \avgint_P 1+ \Phi\left(\frac{|f(x)|}{\lambda}\right)\,dx \right\} \leq 2\|f\|_{\Phi,P}. \]
By the second inequality, we can fix $\lambda_0>0$ such that the
middle quantity is less than
$3\|f\|_{\Phi, P}$.   Then by the first inequality, 
\begin{align*}
 \sum_{\substack{Q\in \D \\ Q\subset P}} 
|Q|^{\frac{\alpha}{n}} |Q|\|f\|_{\Phi,Q}
& = \sum_{k=0}^\infty \sum_{\substack{Q\subset P\\
  \ell(Q)=2^{-k}\ell(P)}}
|Q|^{\frac{\alpha}{n}} |Q|\|f\|_{\Phi,Q} \\
& \leq |P|^{\frac{\alpha}{n}} \sum_{k=0}^\infty  2^{-k\alpha} 
\sum_{\substack{Q\subset P\\
  \ell(Q)=2^{-k}\ell(P)}} 
\lambda_0 \int_Q 1+ \Phi\left(\frac{|f(x)|}{\lambda_0}\right)\,dx
\\
& = C(\alpha) |P|^{\frac{\alpha}{n}}
\lambda_0 \int_P 1+ \Phi\left(\frac{|f(x)|}{\lambda_0}\right)\,dx
\\
& \leq C(\alpha) |P|^{\frac{\alpha}{n}}|P|\|f\|_{\Phi,P}.
\end{align*}

\end{proof}

\begin{proof}
Fix $b\in BMO$. We first make some reductions.   Since $[b,I_\alpha]$ is linear, by splitting $f$ into
its positive and negative parts we may assume $f$ is non-negative.  By
Fatou's lemma we may assume that $f$ is bounded and has compact
support.  Finally, by Proposition~\ref{prop:dyadic-commutator} it will suffice to prove
this result for the dyadic operator $C_b^\D$, where $\D$ is any dyadic
grid.  

We begin by applying duality:  there exists $g\in L^{q'}(u)$,
$\|g\|_{L^{q'}(u)}=1$, such that
\begin{align*}
\|C_b^\D(f\sigma)\|_{L^{q'}(u)}
& = \int_\subRn C_b^\D(f\sigma)g(x)u(x)\,dx \\
& = \sum_{Q\in \D} |Q|^{\frac{\alpha}{n}}\int_Q \avgint_Q
|b(x)-b(y)|f(y)\sigma(y)\,dy\, g(x)u(x)\,dx \\
& \leq \sum_{Q\in \D} |Q|^{\frac{\alpha}{n}} \avgint_Q
|b(x)-\avg{b}_Q|g(x)u(x)\,dx\, \avg{f\sigma}_Q|Q| \\
& \qquad + \sum_{Q\in \D} |Q|^{\frac{\alpha}{n}}
\avgint_Q |b(y)-\avg{b}_Q|f(y)\sigma(y)\,dy \, \avg{gu}_Q |Q|.
\end{align*}

We will estimate the first term; the estimate for the second is
exactly the same, exchanging the roles of $f,\,\sigma$  and $g,\,u$.
Arguing as we did in the proof of Theorem~\ref{thm:frac-testing}, if
we let $\D_N$ be the set of all dyadic cubes $Q$ in $\D$ with
$\ell(Q)=2^N$, then it will suffice to bound this sum with $\D$
replaced by $\D_N$ and with a constant independent of $N$.   Form the
corona decomposition of $f\sigma$ with respect to Lebesgue measure for
each cube in $\D_N$.   Let $\F$ denote the
union of all these cubes.

Let $\Phi(t)=t\log(e+t)$; then $\bar{\Phi}(t)\approx e^t-1$, and so by
the generalized H\"older's inequality and the John-Nirenberg inequality,
\[ \avgint_Q |b(x)-\avg{b}_Q|g(x)u(x)\,dx 
\leq 2\|b-\avg{b}_Q\|_{\bar{\Phi},Q}\|gu\|_{\Phi,Q} 
\leq C(n)\|b\|_{BMO}\|gu\|_{\Phi,Q}. \]
Furthermore, if we define
\[ C(t) = \frac{t^{q'}}{\log(e+t)^{1+(q'-1)\delta}}, \]
then $C\in B_{q'}$ and $A^{-1}(t)C^{-1}(t) \lesssim \Phi^{-1}(t)$.
(See~\cite[Lemma~2.12]{cruz-moen2012} for the details of this calculation.)
We also have that $\bar{B}\in B_p$.  

If we combine all of these facts and use Lemma~\ref{lemma:sum-alpha}
and the generalized H\"older's inequality twice, we get that
\begin{align*}
& 
\sum_{Q\in \D_N} |Q|^{\frac{\alpha}{n}} \avgint_Q
|b(x)-\avg{b}_Q|g(x)u(x)\,dx\, \avg{f\sigma}_Q|Q|  \\
& \qquad \qquad 
\leq C(n) \|b\|_{BMO} \sum_{F\in \F} \avg{f\sigma}_F 
\sum_{\pi_\F(Q)=F} |Q|^{\frac{\alpha}{n}} |Q| \|gu\|_{\Phi,Q} \\
& \qquad \qquad 
\leq C(n,\alpha) \|b\|_{BMO} \sum_{F\in \F} \avg{f\sigma}_F 
|F|^{\frac{\alpha}{n}} |F| \|gu\|_{\Phi,F} \\
& \qquad \qquad 
\leq C(n,\alpha) \|b\|_{BMO} \sum_{F\in \F} 
|F|^{\frac{\alpha}{n}} \|u^{\frac{1}{q}}\|_{A,F}
\|\sigma^{\frac{1}{p'}}\|_{B,F}
 \|f\sigma^{\frac{1}{p}}\|_{\bar{B},F} \|gu^{\frac{1}{q'}}\|_{C,F}
   |E_\F(F)| \\
& \qquad \qquad 
\leq C(n,\alpha) \|b\|_{BMO} [u,\sigma]_{A_{p,q,A,B}^\alpha}\sum_{F\in
  \F} 
\|f\sigma^{\frac{1}{p}}\|_{\bar{B},F} \|gu^{\frac{1}{q'}}\|_{C,F}
   |E_\F(F)|^{\frac{1}{p}+\frac{1}{q'}}. \\
 \end{align*}
We can now apply H\"older's inequality and use the fact that
$\bar{B}\in B_p$ and $C\in B_{q'}$ to finish the argument exactly as
we did in the proof of Theorem~\ref{thm:frac-joined}.
\end{proof}

\section{Separated bump conditions}
\label{section:separated}

We conclude with a discussion of some very recent work and some
additional open problems for fractional integral operators and their
commutators.   To put these into context, we will first
review the Muckenhoupt-Wheeden conjectures for singular integral
operators and their relation to bump conditions.  For a more detailed
overview of these conjectures, see~\cite{dcu-martell-perez,MR2797562}.

\subsection*{The Muckenhoupt-Wheeden conjectures}
 In the late 1970's while studying two weight norm inequalities for the
Hilbert transform, Muckenhoupt and Wheeden made a series of
conjectures relating this problem to two weight norm inequalities for
the maximal operator.(\footnote{I first learned these conjectures from
  P\'erez, and later learned some of their history 
directly from Muckenhoupt.   However, they do not
appear to have ever been published until they appeared
in~\cite{MR2797562}.  The weak $(1,1)$ conjecture appeared shortly
before this in~\cite{MR2427454}.})  These conjectures were quickly
extended to general singular integral operators.   Restated in terms
of weights $(u,\sigma)$ instead of weights $(u,v)$ as they were
originally framed, they conjectured that for $1<p<\infty$, a sufficient condition for a
singular integral operator to satisfy $T(\cdot \sigma) :
L^p(\sigma)\rightarrow L^p(u)$ is that the maximal operator satisfy
\[ M(\cdot \sigma) :L^p(\sigma)\rightarrow L^p(u) \]
and the dual inequality
\[ M(\cdot u) : L^{p'}(u) \rightarrow L^{p'}(\sigma). \]
They further conjectured that the weak type inequality $T :
L^p(\sigma)\rightarrow L^{p,\infty}(u)$ holds if the maximal operator
only satisfies the dual inequality.   (Note the parallels between
these conjectures and the testing conditions described in
Section~\ref{section:testing}.)  Finally, they conjectured that the
following weak $(1,1)$ inequality holds:
\[ \sup_{t>0} t\, u(\{x\in \R^n : |Tf(x)|>t\}) 
\leq C\int_\subRn |f(x)|Mu(x)\,dx. \]

In the one weight case (i.e., with Muckenhoupt $A_p$ weights) all of
these conjectures are true, and with additional assumptions on the
weights (e.g., $u,\,v\in A_\infty$) they are true in the two weight
case.  However, all three conjectures
were recently shown to be false.  The weak $(1,1)$ conjecture was disproved by Reguera and
Thiele~\cite{reguera-thieleP}; the strong $(p,p)$ conjecture by
Reguera and Scurry~\cite{MR3020857}; and building on this the
weak $(p,p)$ conjecture was disproved in~\cite{CRV2012}. 

On the other hand, an ``off-diagonal'' version of this conjecture is
true~\cite{cruz-martell-perezP}:  if $1<p<q<\infty$,  and the maximal operator satisfies
\[ M(\cdot \sigma) :L^p(\sigma)\rightarrow L^q(u) \]
and the dual inequality
\[ M(\cdot u) : L^{q'}(u) \rightarrow L^{p'}(\sigma), \]
then $T : L^p(\sigma)\rightarrow L^q(u)$.    If the dual inequality
holds, then the weak $(p,q)$ inequality $T : L^p(\sigma)\rightarrow
L^{q,\infty}(u)$ holds as well.   Examples of such weights can be
easily constructed using the Sawyer testing condition
(Theorem~\ref{thm:max-testing} with $\alpha=0$).  For instance,
$u=\chi_{[0,1]}$ and $\sigma=\chi_{[2,3]}$ work for all $p>1$.

It follows from Theorem~\ref{thm:strong-max-bump} (with
$\alpha=0$) that these two off-diagonal inequalities for the maximal operator are implied by a pair of
separated bump conditions, $[u,\sigma]_{A_{p,q,B}^0},\,
[u,\sigma]_{A_{p,q,A}^0}^*<\infty$.   When $p=q$ this leads to the
separated bump conjectures for singular integrals:  if  $[u,\sigma]_{A_{p,p,B}^0},\,
[u,\sigma]_{A_{p,p,A}^0}^*<\infty$, then a singular integral satisfies
the strong $(p,p)$ inequality, and if the dual condition holds, it
satisfies the weak $(p,p)$ inequality.  This conjecture is due to P\'erez:  his study of
bump conditions was partly motivated by the Muckenhoupt-Wheeden conjectures.
It was first published, however, in~\cite{CRV2012}, where it was
proved for log bumps: $A(t)=t^p\log(e+t)^{p-1+\delta}$,
$B(t)=t^{p'}\log(e+t)^{p'-1+\delta}$, $\delta>0$ and some closely
related bump conditions (the so called ``loglog'' bumps). 
The proof was quite technical, relying on a ``freezing'' argument and
a version of the corona decomposition.  For
another, simpler proof that also holds in spaces of homogeneous type,
see~\cite{anderson-DCU-moen}.  It is not clear if the separated bump conjecture is
true for singular integrals  only assuming bumps that satisfy the
$B_p$ condition.  For very recent work that suggests it may  be
false, see Lacey~\cite{2013arXiv1310.3507L} and Treil and
Volberg~\cite{2014arXiv1408.0385T}.

\subsection*{Separated bump conditions for fractional integral operators}
Though never addressed by Muckenhoupt and Wheeden, their conjectures
for singular integrals extend naturally to fractional integrals as
well.   Such a generalization was first considered by Carro, {\em et
  al.}~\cite{MR2151228},
 who showed that the analog of the
Muckenhoupt weak $(1,1)$ conjecture,
\begin{equation} \label{eqn:weak-11-frac}
 \sup_{t>0} t\, u(\{x\in \R^n : |I_\alpha f(x)|>t\}) 
\leq C\int_\subRn |f(x)|M_\alpha u(x)\,dx, 
\end{equation}
is false.  

In~\cite{MR3065302} we made the following conjectures:  given
$0<\alpha<n$ and $1<p\leq q < \infty$, suppose the fractional maximal
operator satisfies
\begin{equation} \label{eqn:max-MW}
 M_\alpha(\cdot \sigma) :L^p(\sigma)\rightarrow L^q(u) 
\end{equation}
and the dual inequality
\begin{equation} \label{eqn:max-MW-dual}
 M_\alpha(\cdot u) : L^{q'}(u) \rightarrow L^{p'}(\sigma). 
\end{equation}
Then the strong $(p,q)$ inequality holds, and     if the dual inequality
holds, the weak $(p,q)$ inequality holds.  Analogous to the case of singular integrals,
both of these conjectures are true in when $p<q$:  this was proved
in~\cite{MR3065302}.  Earlier, in~\cite{MR3224572} we proved a weaker version of this
conjecture for separated bump conditions when $\frac{1}{p}-\frac{1}{q}
\approx \frac{\alpha}{n}$. 

\begin{theorem} \label{thm:MW-offdiag}
Given $0<\alpha<n$ and $1<p<q<\infty$, suppose the pair of weights
$(u,\sigma)$ are such that~\eqref{eqn:max-MW}
and~\eqref{eqn:max-MW-dual} hold.  Then $I_\alpha  :
L^p(\sigma)\rightarrow L^q(u)$.  If~\eqref{eqn:max-MW-dual} holds,
then $I_\alpha : L^p(\sigma)\rightarrow
L^{q,\infty}(u)$. 
\end{theorem}

\begin{proof}
It will suffice to prove this for the dyadic fractional integral
operator $I_\alpha^\D$, where $\D$ is any dyadic grid.  We will show
that the desired inequalities follow immediately from Theorem~\ref{thm:lsut-global}.  To see
this we will first consider the testing condition
\[ \I_{\D,out}  = \sup_Q \sigma(Q)^{-\frac{1}{p}}
\left(\int_\subRn I_{\alpha,Q}^{\D,out}(\sigma\chi_Q)(x)^qu(x)\,dx\right)^{\frac{1}{q}}
< \infty. \]
 Fix a cube $Q$ and $x\in \R^n$ such that there exists a dyadic
 cube $P\in \D$ with $x\in P$ and $Q \subsetneq P$.  (If
 no such cube exists then $I_{\alpha,Q}^{\D,out}(\sigma\chi_Q)(x)=0$.) 
Let $Q_0$ be the smallest such cube, and for $k\geq 1$ let $Q_k$ be
the unique dyadic cube
such that $Q_0\subset Q_k$ and $\ell(Q_k)=2^k\ell(Q_0)$.     Then
\begin{multline*}
I_{\alpha,Q}^{\D,out}(\sigma\chi_Q)(x) 
=\sum_{k=0}^\infty
|Q_k|^{\frac{\alpha}{n}}\avg{\sigma\chi_Q}_{Q_k} \chi_{Q_k}(x) \\
=  |Q_0|^{\frac{\alpha}{n}}\avg{\sigma\chi_Q}_{Q_0} |\sum_{k=0}^\infty
2^{k(\alpha-n)} 
\leq C(n,\alpha) |Q_0|^{\frac{\alpha}{n}}\avg{\sigma\chi_Q}_{Q_0}
\leq C(n,\alpha) M_\alpha(\sigma\chi_Q)(x).
\end{multline*}
Therefore, we can replace   $I_{\alpha,Q}^{\D,out}$ by $M_\alpha$
in the testing condition, and if~\eqref{eqn:max-MW} holds, then we immediately get that
$I_{\D,out}<\infty$.    Similarly, if we
assume~\eqref{eqn:max-MW-dual}, then we get that the dual testing
condition satisfies $I_{\D,out}^*<\infty$.   The strong and weak type
inequalities then follow from Theorem~\ref{thm:lsut-global}.
\end{proof}

We do not know whether Theorem~\ref{thm:MW-offdiag} is true when
$p=q$, though the failure of the Muckenhoupt-Wheeden conjectures for
singular integrals suggests that it is false.  However, it is not
clear where to look for a counter-example.  One possibility is to
modify the example of Reguera and Scurry~\cite{MR3020857}.
However, this example depends strongly on the cancellation in the
Hilbert transform, which is not present in the fractional integral,
and it is not certain how this would affect the example.  An
alternative would be to consider the counter-example
to~\eqref{eqn:weak-11-frac} in~\cite{MR2151228}.

\bigskip

When $p=q$ there is a weaker 
conjecture that we believe is true.  As we noted above,  by
Theorem~\ref{thm:strong-max-bump} we have that \eqref{eqn:max-MW}
holds if $[u,\sigma]_{A^\alpha_{p,p,B}}<\infty$, $\bar{B}\in B_p$, and \eqref{eqn:max-MW-dual}
holds if $[u,\sigma]_{A^\alpha_{p,p,A}}^*<\infty$, $\bar{A}\in
B_{p'}$.   We therefore conjecture that that if
$[u,\sigma]_{A^\alpha_{p,p,B}},\,[u,\sigma]_{A^\alpha_{p,p,A}}^*<\infty$,
then $I_{\alpha}(\cdot\sigma) : L^p(\sigma)\rightarrow L^p(u)$, and if
$[u,\sigma]_{A^\alpha_{p,p,A}}^*<\infty$, then $I_{\alpha}(\cdot\sigma) :
L^p(\sigma)\rightarrow L^{p,\infty}(u)$.  

This conjecture is the analog of the separated bump conjecture for
singular integrals.  
For fractional integrals, this conjecture is only known
for ``double'' log bumps: i.e., $A(t)=t^p\log(e+t)^{2p-1+\delta}$,
$B(t)=t^{p'}\log(e+t)^{2p'-1+\delta}$.
In~\cite[Theorem~9.42]{MR2797562} it was shown that for this choice of
$A$ the weak $(p,p)$
inequality is true if $[u,\sigma]_{A^\alpha_{p,p,A}}^*<\infty$.  We
therefore also have that the weak $(p',p')$ inequality is true if
$[u,\sigma]_{A^\alpha_{p,p,B}}<\infty$.   Then by
Corollary~\ref{cor:weak-weak-strong} we have that the two bump
conditions together imply the strong $(p,p)$ inequality.   

The proof that the bump condition implies the weak type inequality has
two steps.  First, using a sharp function estimate and two weight
extrapolation, we prove a weak $(1,1)$ inequality similar
to~\eqref{eqn:weak-11-frac}:
\[  \sup_{t>0} t\, u(\{x\in \R^n : |I_\alpha f(x)|>t\}) 
\leq C\int_\subRn |f(x)|M_{\Phi,\alpha} u(x)\,dx, \]
where $\Phi(t)=t\log(e+t)^{1+\epsilon}$,
$\epsilon>0$, and $M_{\Phi,\alpha}$ is the Orlicz fractional maximal
operator 
\[ M_{\Phi,\alpha} u(x) = \sup_Q
|Q|^{\frac{\alpha}{n}}\|u\|_{\Phi,Q}\, \chi_Q(x). \] 
The weak $(p,p)$ inequality then follows by again applying two weight
extrapolation and a two weight norm inequality for $M_{\Phi,\alpha}$.

We conjecture that the weak $(1,1)$ inequality is true if we replace
$\Phi$ with $\Psi(t)=t\log(e+t)^\epsilon$.  If this were the case,
then the same extrapolation argument would yield the weak $(p,p)$
inequality for log bumps, and the strong type inequality would follow
as before.  The analogous weak $(1,1)$ inequality is true for singular
integrals: this was proved by P\'erez~\cite{perez94b}.  (See
also~\cite{anderson-DCU-moen}.)  Unfortunately, every attempt to adapt
these proofs to fractional integrals has failed.

An alternate approach would be to prove the weak $(p,p)$ inequality
directly using the testing conditions in
Theorem~\ref{thm:frac-testing}.   One way to do this would be to adapt the corona
decomposition argument used in~\cite{CRV2012} to fractional
integrals. We tried to do this, but our proof in~\cite{MR3224572} 
only worked if $\frac{1}{p}-\frac{1}{q} \approx \frac{\alpha}{n}$.
More recently, we have shown~\cite{dcu-km-P} that it can be modified to work provided
$p<q$; but again the argument fails when $p=q$. 
We strongly believe that the separated bump conjecture is true for
log bumps, and suspect that it is true in general.  However, it is
clear that either new ideas or a non-trivial adaptation of existing ones will be
needed to prove it.

\subsection*{Two conjectures for commutators}

We conclude with two conjectures for commutators of
fractional integrals.  The first is a separated bump conjecture.    A
close examination of the proof of Theorem~\ref{thm:comm-joined} shows
that we actually proved something stronger:  we showed that for
$0<\alpha<n$ and $1<p\leq q<\infty$,  if a pair of weights
$(u,\sigma)$ satisfies
\begin{equation} \label{eqn:bump-comm}
\sup_Q|Q|^{\frac{\alpha}{n}+\frac{1}{q}-\frac{1}{p}}\|u^{\frac{1}{q}}\|_{A,Q} \|\sigma^{\frac{1}{p'}}\|_{B,Q} <
\infty, 
\end{equation}
with  $A(t)=t^{q}\log(e+t)^{2q-1+\delta}$, and $B(t)=
t^{p'}\log(e+t)^{p'-1+\delta}$, and
\begin{equation} \label{eqn:dual-bump-comm-dual}
 \sup_Q|Q|^{\frac{\alpha}{n}+\frac{1}{q}-\frac{1}{p}}\|u^{\frac{1}{q}}\|_{C,Q} \|\sigma^{\frac{1}{p'}}\|_{D,Q} <
\infty, 
\end{equation}
with $C(t)=t^{q}\log(e+t)^{q-1+\delta}$, and $D(t)=
t^{p'}\log(e+t)^{2p'-1+\delta}$, 
then the strong $(p,q)$ inequality $[b,I_\alpha](\cdot\sigma) :
L^p(\sigma) \rightarrow L^{q}(u)$ holds.

There is no comparable result known for the weak $(p,q)$ inequality.
However, in~\cite{MR1793688} two weight weak type inequalities were proved for
singular integral operators and  we believe that the proofs
there could be adapted to prove that $[b,I_\alpha](\cdot\sigma) :
L^p(\sigma) \rightarrow L^{q,\infty}(u)$ provided
that~\eqref{eqn:dual-bump-comm-dual} holds with
$C(t)=t^{rq}$,   $r>1$, and $D(t)= t^{p'}\log(e+t)^{p'}$.
Further, using ideas from~\cite{cruz-uribe-fiorenza02}, we could in
fact take  $C$ to be from a family of  Young functions called exponential log
bumps. 

We conjecture that the following separated bump conditions are
sufficient: the strong $(p,q)$ inequality holds if $(u,\sigma)$
satisfy~\eqref{eqn:bump-comm} and~\eqref{eqn:dual-bump-comm-dual} but
with $B(t)=t^{p'}$ and $C(t)=t^{q'}$.  Similarly, the weak $(p,q)$
inequality holds if~\eqref{eqn:bump-comm} holds with $B(t)=t^{q'}$.
To prove these conjectures, it would suffice to prove them for the dyadic
operator $C_b^\D$ in Proposition~\ref{prop:dyadic-commutator}. 
It will probably 
be easier to prove these conjectures in the off-diagonal case when
$p<q$.  One approach in this case would be to prove a ``global''
version of the testing condition conjectures for commutators given at
the end of Section~\ref{section:testing}.  This might be done by
adapting the arguments in~\cite{LacSawUT2010}.  
Further,  though
it would probably not yield the full conjecture, it would be 
interesting to see if the proof  in~\cite[Theorem~9.42]{MR2797562} for
fractional integrals could be modified to prove a non-optimal weak
type inequality for commutators.  It seems possible that this approach
would yield the weak type inequality with
$A(t)=t^{q}\log(e+t)^{3q-1+\delta}$.  

\medskip

The second conjecture concerns the necessity of $BMO$ for commutators
to be bounded.  In~\cite{MR642611} Chanillo showed that  if $[b,I_\alpha] :
L^p\rightarrow L^q$, $\frac{1}{p}-\frac{1}{q}=\frac{\alpha}{n}$, and
$n-\alpha$ is an even integer, then $b\in BMO$.  (Very recently, this
restriction was removed by Chaffee~\cite{2014arXiv1410.4587C}.) At the end of the meeting in Antequera, J.~L. Torrea asked
if anything could be said about $b$ if
there exists a pair of weights $(u,\sigma)$ (or perhaps a family of
such pairs) such that $[b,I_\alpha](\cdot\sigma) :
L^p(\sigma)\rightarrow L^q(u)$.    Nothing is known about this
question, but it merits further investigation.  

\bibliographystyle{plain}
\bibliography{antequera}

\end{document}